\documentclass[
	12pt,oneside,english,a4paper]{amsart}
\usepackage[T1]{fontenc}
\usepackage[utf8]{inputenc}
\usepackage[dvipsnames]{xcolor}
\usepackage[yyyymmdd,hhmmss]{datetime}
\usepackage[numbers,sort,compress]{natbib}
\usepackage{amstext,amsthm,amssymb}
\usepackage{stmaryrd}
\usepackage{xparse,mathrsfs,mathtools}
\usepackage{tikz}
\usetikzlibrary{backgrounds,patterns}
\usepackage[english]{babel}
\usepackage{xr-hyper}
\usepackage[pdftex,unicode=true,%
	pdfusetitle,pdfdisplaydoctitle=true,%
	pdfpagemode=UseOutlines,%
	bookmarks=true,bookmarksnumbered=false,bookmarksopen=true,bookmarksopenlevel=2,%
	breaklinks=true,%
	pdfencoding=unicode,psdextra,
	pdfcreator={},
	pdfborder={0 0 0},
	colorlinks=false
]{hyperref}
\usepackage{lmodern,microtype}
\usepackage{subcaption}
\usepackage[dvipsnames]{xcolor}
\usepackage{geometry}
\usepackage[capitalise,nameinlink]{cleveref}
\usepackage[cal=zapfc,bb=ams,frak=esstix,scr=boondoxo]{mathalfa}

\hypersetup{
	colorlinks=false,
	pdfborder={0 0 0},
	pdfborderstyle={/S/U/W 0},
}

\newtheorem*{thm*}{Theorem}
\newtheorem{thm}{Theorem}[section]
\newtheorem{lem}[thm]{Lemma}
\newtheorem{cor}[thm]{Corollary}
\newtheorem{prop}[thm]{Proposition}
\theoremstyle{remark}

\newtheorem*{rem*}{Remark}

\newcommand{\NN}{\mathbb{N}}
\newcommand{\ZZ}{\mathbb{Z}}
\newcommand{\QQ}{\mathbb{Q}}
\newcommand{\RR}{\mathbb{R}}

\newcommand{\dd}[1]{\mathop{\mathrm{d}#1}}
\NewDocumentCommand\e{ s O{} m }{%
	\IfBooleanTF{#1}{%
		\operatorname{e}_{#2}\parentheses*{#3}%
	}{\operatorname{e}_{#2}\parentheses{#3}}%
}

\newcommand{\LandauO}{O}

\DeclarePairedDelimiter\parentheses{\lparen}{\rparen}
\DeclarePairedDelimiter\braces{\lbrace}{\rbrace}
\DeclarePairedDelimiter\brackets{\lbrack}{\rbrack}
\DeclarePairedDelimiter\abs{\lvert}{\rvert}

\DeclarePairedDelimiter\floor{\lfloor}{\rfloor}
\DeclarePairedDelimiter\ceil{\lceil}{\rceil}

\DeclarePairedDelimiter\ropeninterval{\lbrack}{\rparen}
\DeclarePairedDelimiter\rcfrac{\llbracket}{\rrbracket} 
\NewDocumentCommand\set{ s o m o }{%
	\IfBooleanTF{#1}{\IfNoValueTF{#4}{\braces*{#3}}{\braces*{\,#3:#4\,}}}{%
	\IfNoValueTF{#2}{\IfNoValueTF{#4}{\braces{#3}}{\braces{\,#3:#4\,}}}{%
	\IfNoValueTF{#4}{\braces[#2]{#3}}{\braces[#2]{\,#3:#4\,}}}}%
}

\newcommand{\sawtooth}[1]{(\mkern -2mu(#1)\mkern -2mu)}%
\DeclareMathOperator{\inv}{inv}
\newcommand{\EAclassical}{\ensuremath{ \mathrm{EA}^{(\mathrm{sub})} }}
\newcommand{\EAdivision}{\ensuremath{ \mathrm{EA}^{(\mathrm{div})} }}
\newcommand{\EAbyexcess}{\ensuremath{ \mathrm{EA}^{(\mathrm{div})}_{(\textnormal{by-excess})} }}

\crefname{section}{§}{§§}
\crefname{figure}{Figure}{Figures}
\crefformat{equation}{#2(#1)#3}
\crefformat{enumi}{#2(#1)#3}
\numberwithin{equation}{section}
\makeatletter
	\renewcommand\p@subfigure{\thefigure~}
\makeatother

\usepackage{xifthen,ifthen}
\makeatletter
\newcounter{@ToDo}
\newcommand{\todo@helper}[1]{%
	({\color{blue}TODO~\arabic{@ToDo}: {#1\@addpunct{.}}})%
}
\newcommand{\todo}[1]{\stepcounter{@ToDo}%
	\relax\ifmmode\text{\todo@helper{#1}}%
	\else\todo@helper{#1}\fi%
}
\makeatother

\title[Bias in the Euclidean algorithm and a conjecture of Ito]{Bias in the number of steps in the Euclidean algorithm and a conjecture of Ito on Dedekind sums}
\date{\today{}}

\makeatletter\@namedef{subjclassname@2020}{\textup{2020} Mathematics Subject Classification}\makeatother

\author{Paolo~Minelli}

\author{Athanasios~Sourmelidis}

\author{Marc~Technau}

\address{
	Paolo~Minelli \and Athanasios~Sourmelidis \and Marc~Technau\\%
	Institut für Analysis und Zahlentheorie\\%
	TU~Graz\\%
	Kopernikusgasse~24/II\\%
	8010~Graz\\%
	Austria}
\email{minelli@math.tugraz.at}
\email{sourmelidis@math.tugraz.at}
\email{mtechnau@math.tugraz.at}

\thanks{%
	PM is supported by the Austrian Science Fund (FWF), project~I-3466.
	AS is supported by FWF projects~Y-901 and~F-5512.
	MT is supported by the joint FWF--ANR project \emph{ArithRand} (FWF I 4945-N and ANR-20-CE91-0006).
}

\begin{document}
\begin{abstract}
	We investigate the number of steps taken by three variants of the Euclidean algorithm on average over Farey fractions.
	We show asymptotic formulae for these averages restricted to the interval $(0,1/2)$, establishing that they behave differently on $(0,1/2)$ than they do on $(1/2,1)$.
	These results are tightly linked with the distribution of lengths of certain continued fraction expansions as well as the distribution of the involved partial quotients.
	
	As an application, we prove a conjecture of Ito on the distribution of values of Dedekind sums.
	
	The main argument is based on earlier work of Zhabitskaya, Ustinov, Bykovski\u{\i} and others, ultimately dating back to Heilbronn, relating the quantities in question to counting solutions to a certain system of Diophantine inequalities.
	The above restriction to only half of the Farey fractions introduces additional complications.
\end{abstract}
\maketitle%
\section{Introduction}

\subsection{Euclidean algorithm (classical version)}
\label{sec:Euclid:Classical}

The Euclidean algorithm---ref\-er\-red to as `\EAclassical' in the sequel---for the computation of the greatest common divisor (gcd) of two positive integers $a$ and $b$, has been described as \emph{`the oldest non-trivial algorithm that has survived to the present day'} by Knuth~\cite[p.~318]{knuth1998art-of-programming-2}.
In its most basic form the algorithm proceeds by replacing the input tuple $(a,b)$ by $(a-b,b)$ if $a<b$ (`Case~A') and $(a,b-a)$ if $a\geq b$ (`Case~B') until one of the arguments becomes zero (`Case~C'), in which case the gcd of the original input is given by the other argument.
(There is some leeway in describing the algorithm and we shall choose what is convenient for our exposition rather than what is historically most accurate; the reader is referred to \emph{loc.~cit.}\ for a more detailed discussion of that matter.)
For instance, on the input $(11,3)$, the algorithm takes the following~$6$ steps:
\begin{equation}\label{eq:Example}
	\begin{aligned}
		(11,3) &
		\mapsto (8,3)
		\mapsto (5,3)
		\stackrel{*}{\mapsto} (2,3)
		\stackrel{*}{\mapsto} (2,1) \\ &
		\mapsto (1,1)
		\stackrel{*}{\mapsto} (\underline{1},0)
		\quad(\text{hence, }\gcd(11,3) = \underline{1}),
	\end{aligned}
\end{equation}
where the asterisks ($*$) mark the positions where the algorithm switches between cases.
Observe that the number $11/3$ has the continued fraction expansion
\begin{equation}\label{eq:ContFracExample}
	\frac{11}{3} = 3 + \cfrac{1}{ 1 + \cfrac{1}{ 2 } } \, .
\end{equation}
and~$6 = 3 + 1 + 2$ is the sum of the partial quotients herein.

If one modifies Case~A of \EAclassical{} as to replace $(a,b)$ by $(a-B,b)$, where $B$ is the largest multiple of $b$ not exceeding $a$, and modifies Case~B similarly, then the modified algorithm skips all steps ($\mapsto$) not marked with an asterisk in the above example; this amounts to precisely~$3$ steps which is also the number of partial quotients in the continued fraction expansion~\cref{eq:ContFracExample};
we shall refer to this version of \EAclassical{} by \EAdivision{}.

It is easy to see that the correspondence of number of steps on the input $(a,b)$ and properties of the continued fraction expansion
\begin{gather}\label{eq:OrdinaryContinuedFractionExpansion}
	\frac{a}{b}
	= [0; a_1, \ldots, a_n]
	\coloneqq 0 + \cfrac{1}{
		a_1 + \cfrac{1}{
			a_2 + \ldots \cfrac{}{
					\ldots + \cfrac{1}{
					a_n
				}
			}
		}
	} 
\end{gather}
of~$a/b\in[0,1)$ (where $n\in\NN_0$ and the so-called \emph{partial quotients} $a_1,a_2,\ldots,a_n$ are positive integers and $a_n\geq 2$) holds in general, \emph{i.e.},
\begin{itemize}
	\item the number of steps taken by \EAclassical{} when applied to $(a,b)$ (or any tuple $(ka,kb)$ with some positive integer $k$) is $a_1 + a_2 + \ldots + a_n$ (see \cref{fig:Euclid:Subtractive} for a plot of its behavior), and
	\item the number of steps taken by \EAdivision{} is $n$.
	We denote this number by $s(a/b)$. (See \cref{fig:Euclid:Division} for a plot of its behavior.)
\end{itemize}

\begin{figure}[!ht]
	\centering%
	\newlength{\subFigureWidth}
	\setlength{\subFigureWidth}{6.9cm}
	\begin{subfigure}[t]{\subFigureWidth}
		\centering%
		\includegraphics{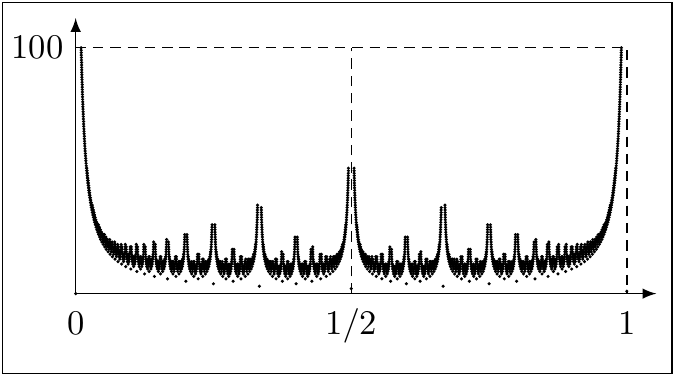}%
		\caption{%
			\EAclassical{}.
			The maximum number of steps occurs at $1/100$ and $99/100$ which have the continued fraction expansion $[0; 100]$ and $[0; 1, 99]$ respectively.
		}%
		\label{fig:Euclid:Subtractive}%
	\end{subfigure}%
	\qquad%
	\begin{subfigure}[t]{\subFigureWidth}%
		\centering%
		\includegraphics{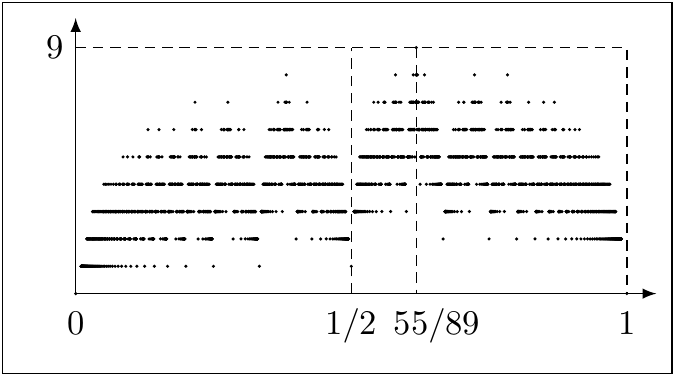}%
		\caption{%
			\EAdivision{}.
			The maximum number of steps occurs at $55/89$ which has continued fraction expansion $[0; 1, 1, 1, 1, 1, 1, 1, 1, 2]$.
		}%
		\label{fig:Euclid:Division}%
	\end{subfigure}%
	\\[4mm]
	\begin{subfigure}[t]{6.6cm}
		\centering%
		\includegraphics{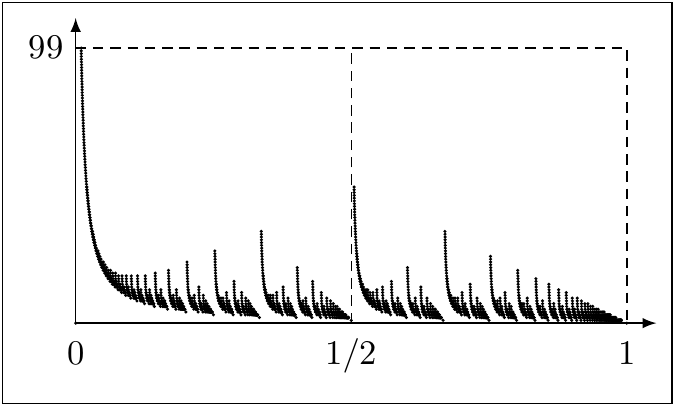}%
	\end{subfigure}%
	\qquad%
	\begin{subfigure}[b]{\subFigureWidth}%
		\caption{%
			\EAbyexcess{} (defined in~\cref{sec:Euclid:Variant}).
			The maximum number of steps occurs at $1/100$ which has the minus continued fraction expansion $\rcfrac{ 1; 2, \ldots, 2, 2 }$ (with `$2$' occurring $99$ times).
		}%
		\label{fig:EuclidMinus}%
	\end{subfigure}%
	\vspace{2mm}%
	\caption[]{%
		The number of steps of \EAclassical{}, \EAdivision{} \&{} \EAbyexcess{} when applied to all reduced $a/b\in[0,1)\cap\QQ$ with $1\leq b\leq 100$.
	}%
	\label{fig:Euclid}%
\end{figure}

\subsection{Variants of the Euclidean algorithm}
\label{sec:Euclid:Variant}

Several other variants of the Euclidean algorithm have been considered in the literature (see, e.g., \cite{vallee2000a-unifying-framework,Valleedynamical} for a selection).
For the most part, they arise (ignoring some technicalities) from modifying the distinguishing conditions of the cases~A and~B as introduced in~\cref{sec:Euclid:Classical}.
Here we discuss only one such variant.
In fact, for convenience, we restrict our discussion to only stating a variant that is more similar in spirit to \EAdivision{} rather than \EAclassical{}.
To obtain this variant---referred to as \EAbyexcess{} in the sequel---modify Case~A of \EAdivision{} to replace the input $(a,b)$ by $(B-a,b)$, where $B$ is the smallest multiple of $b$ not smaller than $a$ and make a similar modification to Case~B.
Given this modification, our example~\cref{eq:Example} takes the shape
\(
	(11,3) \stackrel{*}{\mapsto} (1,3) \stackrel{*}{\mapsto} (\underline{1},0)
\).

Once more, one can associate a certain continued fraction expansion of a number $a/b\in[0,1)$ to the behaviour of the algorithm on the input $(a,b)$.
The particular continued fraction expansion relevant in this case is often called \emph{minus continued fraction expansion}\footnote{%
	Instead of `minus', some authors use the attribute `backwards' or `regular' instead.
} and takes the shape
\begin{gather}\label{eq:RegularContinuedFractionExpansion}
	\frac{a}{b}
	= \rcfrac{ 1; b_1, \ldots, b_m }
	\coloneqq 1 - \cfrac{1}{
		b_1 - \cfrac{1}{
			b_2 - \ldots \cfrac{}{
					\ldots - \cfrac{1}{
					b_m
				}
			}
		}
	} \, ,
\end{gather}
where $m\in\NN$ and $b_1,b_2,\ldots,b_m\geq 2$ are integers.
When expanding $a/b$ as in~\cref{eq:RegularContinuedFractionExpansion}, then $m+1$ can be seen to be the number of steps taken by \EAbyexcess{} on the input $(a,b)$.
We shall write $\ell(a/b)$ for the number $m$ from~\cref{eq:RegularContinuedFractionExpansion} in the sequel.
(See \cref{fig:EuclidMinus} for a plot of $\ell(a/b)$.)
For further background on continued fractions we refer to~\cite{Perron}.

\subsection{Asymptotics for the number of steps of Euclidean algorithms}

It is an interesting question to study statistical properties of the number of steps of the Euclidean algorithm (and its variants), or---equivalently---distribution properties of continued fractions.
It was Heilbronn~\cite{heilbronn1968average-length} who first identified the principal term of the asymptotics for the average number of steps in the case of the classical Euclidean algorithm, the average being taken over \emph{numerators}:
\[
	\frac{1}{\varphi(b)} \sum_{\substack{ a\leq b \\ \gcd(a,b)=1 }} s\parentheses*{\frac{a}{b}}
	= A_1 \log b + O\parentheses{ (\log\log b)^4 } \quad
	(\text{as~} b\to\infty);
\]
here $\varphi(n) \coloneqq \#\set{ 1\leq m\leq n : \gcd(m,n)=1 }$ ($n\in\mathbb{N}$) is Euler's totient function and $A_1$ is an explicitly given non-zero constant.\footnote{%
	See~\cref{sec:Notation} for a comment on the notation.
}
For the same average, an asymptotic formula with two significant terms was obtained later by Porter~\cite{Porter}:
\[
	\frac{1}{\varphi(b)} \sum_{\substack{ a\leq b \\ \gcd(a,b)=1 }} s\parentheses*{\frac{a}{b}}
	= A_1 \log b + A_2 + O_\epsilon\parentheses{ b^{-1/6+\epsilon} };
\]
here $A_1$ is as before and $A_2$ is also an explicitly given non-zero constant.
Bykovski\u{\i} and Frolenkov~\cite{Frolenkov2} have recently obtained a generalisation of this and obtained an improved error term.

Considering averages over both numerators \emph{and} denominators, an asymptotic formula with power-law fall-off in the error term was obtained by Vall\'ee~\cite{vallee2000a-unifying-framework} through the use of probability theory and ergodic-theoretic methods. 
This was improved by Ustinov~\cite{Ustinovasymptotic}, who obtained an asymptotic formula with better fall-off in the error term than the one that can be derived from Porter's result:
\begin{equation}\label{Ustinov}
	\frac{1}{\#\mathscr{F}(Q)} \mathop{ \sum_{b\leq Q} \sum_{a\leq b} }_{ \gcd(a,b)=1 } s\parentheses*{\frac{a}{b}}
	= B_1 \log Q + B_2 + O\parentheses{ (\log Q)^5 / Q },
\end{equation}
where
\[
	B_1 = \frac{\log 2}{2\zeta(2)}, \quad
	B_2 = \frac{\log 2}{4\zeta(2)} \parentheses*{ 3 \log 2 + 4 \gamma - 2 \frac{\zeta'(2)}{\zeta(2)} - 3 } - \frac{1}{4},
\]
$\gamma$ denotes the Euler--Mascheroni constant, $\zeta$ is the Riemann zeta function, and
\[
	\mathscr{F}(Q) = \set{ a/b\in\QQ }[ \gcd(a,b)=1,\, 0\leq a\leq b\leq Q ]
\]
denotes the set of \emph{Farey fractions of order $Q$}.
In this regard it is worth noting that another natural way of averaging is over all pairs $(a,b)$ with $1\leq a\leq b \leq Q$ without assuming coprimality of $a$ and $b$.
However, this situation is easily covered using~\cref{Ustinov} and Möbius inversion.

\medskip

While examining the statistical properties of different variations of the Euclidean algorithm, Vall\'ee~\cite{Valleedynamical} obtained also the leading term of the asymptotic formula for the expectation of the number of steps of the by-excess Euclidean algorithm (and hence for the average length of minus continued fractions).
This was improved by Zhabitskaya~\cite{zhabitskaya2009average-length} (following the approach of Ustinov~\cite{Ustinovasymptotic}), a few years later, who showed that
\begin{equation}\label{Zhabitskaya}
	\frac{1}{\#\mathscr{F}(Q)} \mathop{ \sum_{b\leq Q} \sum_{a\leq b} }_{ \gcd(a,b)=1 } \ell\parentheses*{\frac{a}{b}}
	= C_1 (\log Q)^2 + C_2 \log Q + C_3 + O\parentheses{ (\log Q)^6 / Q },
\end{equation}
where $C_1, C_2, C_3$ are explicitly given non-zero constants, the first two being given by
\begin{equation}\label{eq:Zhabitskaya:Constants}
	C_1 = \frac{1}{2\zeta(2)}, \quad
	C_2 = \frac{1}{\zeta(2)} \parentheses*{ 2\gamma - \frac{3}{2} - 2 \frac{\zeta'(2)}{\zeta(2)} },
\end{equation}
and the value of $C_3$ being given by a somewhat longer, yet similar expression which we omit here.
Both error terms in~\cref{Ustinov} and~\cref{Zhabitskaya} have been improved to $O\parentheses{ (\log Q)^3 / Q }$ by Frolenkov~\cite{Frolenkov} who incorporated ideas of Selberg from the elementary proof of the prime number theorem.

For more results regarding the expectation and the variance of the number of steps of the classical and by-excess Euclidean algorithm, we also refer to the work of Baladi and Vall\'ee~\cite{BaladiVallee}, Bykovski\u{\i}~\cite{Bykovski}, Dixon~\cite{dixon1970number-of-steps,Dixonasimple}, Hensley~\cite{hensley1994number-of-steps} and Ustinov~\cite{Ustinovcalculation,Ustinovonthestatistical}.

\subsection{Dedekind sums}

Let $\floor{\eta} = \min\set{ n\in\ZZ }[ n\leq\eta ]$ denote the integer part of $\eta\in\RR$.
Then the \emph{saw-tooth function} is defined as
\[
	\sawtooth{\eta} = \begin{cases}
		\eta - \floor{\eta} - 1/2 & \text{if } \eta\in\RR\setminus\ZZ, \\
		0                         & \text{if } \eta\in\ZZ. \\
	\end{cases}
\]
For any pair $a,b\in\ZZ$, $b\neq 0$, the \emph{Dedekind sum}\footnote{
	The notation $s(a/b)$ is also commonly used, but would conflict with our notation for the length of~\cref{eq:OrdinaryContinuedFractionExpansion}.
} $D(a,b)$ is defined as
\[
	D(a,b) = \sum_{n\leq b} \parentheses*{\!\!\parentheses*{\frac{n}{b}}\!\!} \parentheses*{\!\!\parentheses*{\frac{na}{b}}\!\!}.
\]
It can be verified that $D(a,b) = D(ka,kb)$ for any non-zero integer $k$.
Hence, $D(a/b) \coloneqq D(a,b)$ is well defined.
Moreover, the function $D\colon \QQ\to\QQ$ just defined is periodic with period one.

Dedekind sums originally arose in connection with the multiplier system
for Dedekind's \emph{eta} function over the modular group of two by two integer matrices of determinant~one~\cite{Dedekind} and also satisfy a curious reciprocity law.
By means of the latter Barkan~\cite{Barkan} and (independently) Hickerson~\cite{hickerson1977continued-fractions} have obtained the following identity which connects Dedekind sums with continued fraction expansions:
\begin{align}\label{Hickerson}
	D(a/b)
	= \frac{(-1)^n - 1}{8} + \frac{ a/b - (-1)^n \mkern 1mu [0;a_n,\ldots,a_2,a_1] + \varSigma_\pm(a/b) }{12};
\end{align}
here $a/b = [0; a_1,a_2,\ldots,a_n]$ is as in~\cref{eq:OrdinaryContinuedFractionExpansion} and
\begin{equation}\label{eq:varSigmaPM}
	\varSigma_\pm(a/b) \coloneqq \sum_{j\leq n} (-1)^{j-1} a_j.
\end{equation}
(See \cref{fig:Euclid:Alternating} for a plot of $\varSigma_\pm$.)
In particular, Hickerson employed~\cref{Hickerson} to prove that the set $\set{(a/b,D(a/b)) : a/b\in\QQ }$ is dense in $\RR^2$.

\begin{figure}[!ht]
	\centering
	\includegraphics{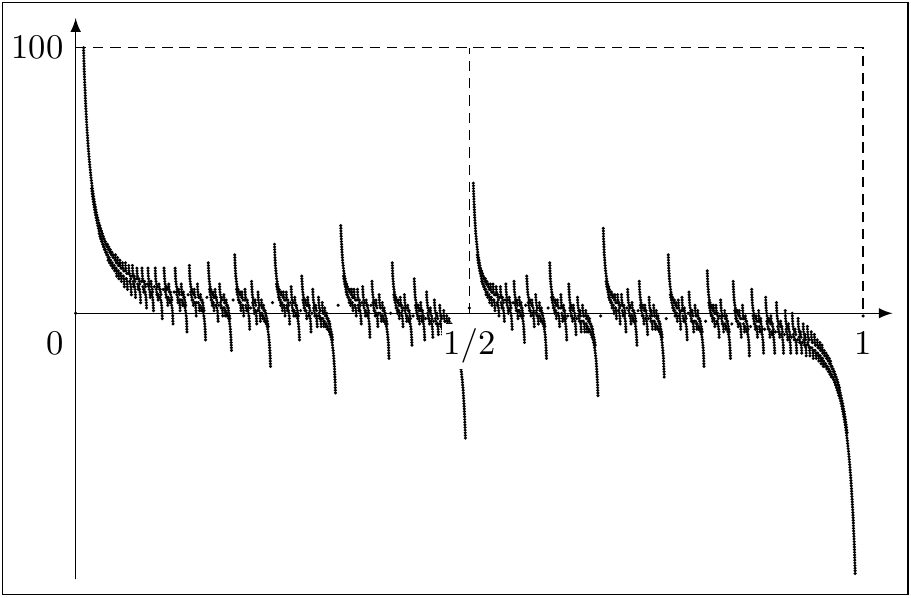}
	\caption[]{%
		Plot of $\varSigma_\pm(a/b)$ when applied to all Farey fractions $a/b\in[0,1]\cap\QQ$ with $1\leq b\leq 100$.
		Note that the average of the plotted values over the interval $[0,1/2)$ is clearly positive, whereas the average of the plotted values over the interval $[1/2,1)$ is negative.
	}%
	\label{fig:Euclid:Alternating}%
\end{figure}

Concerning distribution properties of Dedekind sums observe that via the symmetry property $D(x) = - D(1-x)$ it is easy to see that
\[
	\sum_{x\in\mathscr{F}(Q)} D(x) = 0.
\]
On the other hand, let
$
	\mathscr{F}_0(Q) = \mathscr{F}(Q) \cap \lbrack0,1/2\rparen
$
denote `half' of all Farey fractions with denominators bounded by $Q$.
Then, on the basis of numerical evidence, it has been conjectured by Ito~\cite{ito2004density-result} that
\begin{gather}\label{eq:UnderstandThis}
	\lim\limits_{Q\to\infty} \varSigma(Q) = +\infty,
	\quad\text{where}\quad
	\varSigma(Q) \coloneqq \frac{1}{\#\mathscr{F}(Q)} \sum_{x\in\mathscr{F}_0(Q)} D(x).
\end{gather}
For an exposition of results on Dedekind sums we refer to the classical work of Rademacher and Grosswald~\cite{RademacherGrosswald}, as well as a more up-to-date survey of Girstmair~\cite{girstmair2018on-the-distribution} with a focus on distribution properties.

\section{Main results}

\subsection{Results}

One of the main results of the present work is a proof of Ito's conjecture:
\begin{thm}[Ito's conjecture is true]\label{thm:Goal}
	The statement in \cref{eq:UnderstandThis} holds.
	In fact, one even has the following stronger quantitative version:
	\begin{equation}\label{eq:Ito:AsymptoticVersion}
		\frac{1}{\#\mathscr{F}(Q)} \sum_{x\in\mathscr{F}_0(Q)} D(x)
		= \frac{1}{16} \log Q + O(1).
	\end{equation}
\end{thm}

The proof of \cref{thm:Goal} rests crucially on the following variant of~\cref{Zhabitskaya} which we believe to be of independent interest:
\begin{thm}[Bias in \EAbyexcess{}]%
	\label{thm:Zhabitskaya:RestrictedVariant}
	We have
	\[
		\frac{1}{\#\mathscr{F}(Q)}\sum_{x\in\mathscr{F}_0(Q)} \ell(x)
		= c_1 (\log Q)^2 + c_2 \log Q + O(1),
	\]
	where $c_1,c_2$ are non-zero constants satisfying $2c_1 = C_1$ and $2c_2 > C_2$ with the constants $C_1$ and $C_2$ given in~\cref{eq:Zhabitskaya:Constants}.
	More precisely,
	\[
		c_1 = \frac{1}{4\zeta(2)}, \quad
		c_2
		= \frac{1}{2\zeta(2)} \parentheses*{ 2\gamma - \frac{3}{2} - 2\frac{\zeta'(2)}{\zeta(2)} + \frac{3\zeta(2)}{4} }
		= \frac{C_2}{2} + \frac{3}{8}.
	\]
\end{thm}

The above theorem may be interpreted as a quantitative version of the statement that the length $\ell(a/b)$ of the minus continued fraction expansion~\cref{eq:RegularContinuedFractionExpansion} tends to be larger on average on $\mathscr{F}_0(Q)$ than on $\mathscr{F}(Q) \setminus \mathscr{F}_0(Q)$ (due to $2c_2 > C_2$; see~\cref{Zhabitskaya}).
This may be phrased equivalently as saying that \EAbyexcess{} takes longer on average for fractions in $\lbrack0,1/2\rparen$ than it does for fractions in $\lbrack1/2,1\rparen$.

\medskip

In view of the above it seems natural to ask if similar results can be obtained for the other algorithms \EAclassical{} and \EAdivision{} discussed in~\cref{sec:Euclid:Classical}.
This turns out to be a rather easier question.
For \EAclassical{} one sees no difference in behaviour on $\mathscr{F}_0(Q)$ versus on $\mathscr{F}(Q) \setminus \mathscr{F}_0(Q)$, as should be evident from the symmetry in~\cref{fig:Euclid:Subtractive} about the vertical line through $1/2$.
The latter symmetry may be verified easily by noting that $x = [0; a_1, a_2,\ldots,a_n]$ (with $a_1\geq 2$ so that $x \leq 1/2$) and $1-x = [0; 1, a_1-1, a_2,\ldots,a_n]$ have the same sum of partial quotients, \emph{viz.} identical running time when fed into \EAclassical{}.
On the other hand, an analogue of \cref{thm:Zhabitskaya:RestrictedVariant} may be obtained for \EAdivision{}:
\begin{prop}[Bias in \EAdivision{}]%
	\label{thm:Zhabitskaya:RestrictedVariant:Euclid}
	We have
	\[
		\frac{1}{\#\mathscr{F}(Q)}\sum_{x\in\mathscr{F}_0(Q)} s(x)
		= b_1 \log Q + b_2 + \LandauO((\log Q)^5 / Q),
	\]
	where $2 b_1 = B_1$ and $2b_2 < B_2$ with the constants $B_1$ and $B_2$ given from~\cref{Ustinov}.
	More precisely, $2 b_2 = B_2 - 1/2$.
\end{prop}
\begin{proof}
	This follows immediately from~\cref{Ustinov} and the fact that $s(x) = s(1-x) - 1$ for $x\in(0,1/2)$.
\end{proof}
We should like to mention that Bykovski\u{\i}~\cite{Bykovski} has obtained an asymptotic formula for averaging $s(a/q)$ over all $a$ in some arbitary interval of length at most $q$.
However, the error term in his result does not permit one to deduce \cref{thm:Zhabitskaya:RestrictedVariant:Euclid}.

Generalising \cref{thm:Zhabitskaya:RestrictedVariant} and \cref{thm:Zhabitskaya:RestrictedVariant:Euclid} to averages over $\mathscr{F}\cap\ropeninterval{0,\alpha}$ seems to be an interesting problem.
However, this requires a more careful analysis and a sufficiently flexible generalisation of \cref{lem:Euler'sPhiInHalfIntervals} below.
As this seemed dispensable for our primary intent of proving \cref{thm:Goal}, we shall address this elsewhere in forthcoming work (see also the first author's doctoral dissertation~\cite{minelli2022thesis}).

\subsection{Plan of the paper}

In the next section we show how \cref{thm:Goal} can be deduced from \cref{thm:Zhabitskaya:RestrictedVariant}.
The proof of \cref{thm:Zhabitskaya:RestrictedVariant} is rather more involved.
In~\cref{sec:MainThm:SketchOfProof} we sketch the overall argument and show how \cref{thm:Zhabitskaya:RestrictedVariant} can be deduced from a technical proposition (\cref{prop:CountingSolutionsToSystem}).
The proof of the latter is carried out in~\cref{sec:NumberOfSolutions:RQ}.

\subsection{Notation}
\label{sec:Notation}

We use the Landau notation $f(x) = \LandauO(g(x))$ and the Vinogradov notation $f(x) \ll g(x)$ to mean that there exists some constant $C>0$ such that $\abs{f(x)} \leq C g(x)$ holds for all admissible values of $x$ (where the meaning of `admissible' will be clear from the context).
Unless otherwise indicated, any dependence of $C$ on other parameters is specified using subscripts.
Similarly, we write `$f(x) = o(g(x))$ as $x\to\infty$' if $g(x)$ is positive for all sufficiently large values of $x$ and $f(x)/g(x)$ tends to zero as $x\to\infty$.

Given two coprime integers $a$ and $q\neq 0$ we write $\inv_q(a)$ for the smallest positive integer in the residue class $(a\bmod q)^{-1}$.

\section{Deducing \texorpdfstring{\cref{thm:Goal}}{Theorem\autoref{thm:Goal}} from \texorpdfstring{\cref{thm:Zhabitskaya:RestrictedVariant}}{Theorem\autoref{thm:Zhabitskaya:RestrictedVariant}}}

Throughout this section we shall assume that \cref{thm:Zhabitskaya:RestrictedVariant} has already been proved.
The main tool for deducing \cref{thm:Goal} from \cref{thm:Zhabitskaya:RestrictedVariant} is the formula~\cref{Hickerson} of Barkan and Hickerson.
In this vein, recall also the definition of $\varSigma_\pm(x)$ given in~\cref{eq:varSigmaPM}.
For a number $x\in\lbrack0,1\rparen$ as in~\cref{eq:OrdinaryContinuedFractionExpansion} let
\[
	\varSigma_{\text{odd} }(x) = \sum_{\substack{ i=1 \\ i\text{ odd}  }}^n a_i, \quad
	\varSigma_{\text{even}}(x) = \sum_{\substack{ i=2 \\ i\text{ even} }}^n a_i.
\]
Then, clearly,
\begin{equation}\label{eq:varSigmaPM:Decomposition}
	\varSigma_\pm(x) = \varSigma_{\text{odd}}(x) - \varSigma_{\text{even}}(x).
\end{equation}
The connection with minus continued fraction expansions and, thus, \cref{thm:Zhabitskaya:RestrictedVariant} arises as follows:
in~\cite{zhabitskaya2011mean-value} Zhabitskaya notes\footnote{%
	There appears to be a misprint in~\cite[Eq.~(8)]{zhabitskaya2011mean-value}: the left hand side should read $l'((b-a)/b)$, as can be deduced from the equations~(5) and~(7) in \emph{loc.\ cit.}
} that it is implicit in an article of Myerson~\cite{myerson1987semi-regular} that
\begin{gather}
	\label{eq:Length:Vs:OddPartialQuotientSum}
	\ell( x ) = \varSigma_{\text{odd} }(x) - \epsilon(x), \\
	\label{eq:Length:Vs:EvenPartialQuotientSum}
	\ell(1-x) = \varSigma_{\text{even}}(x) + \epsilon(x).
\end{gather}
Here $\epsilon(x) \in \set{0,1}$ is some correction term which is related to our way of forcing uniqueness in the continued fraction expansion~\cref{eq:OrdinaryContinuedFractionExpansion} by means of requiring the last partial quotient $a_n$ to exceed $1$.
In fact, one can describe the value of $\epsilon(x)$ quite precisely (see~\cite{zhabitskaya2011mean-value}), but this is not necessary for our particular application.

\begin{cor}\label{cor:SigmaPmAsymptotics}
	We have
	\[
		\frac{1}{\#\mathscr{F}(Q)}\sum_{x\in\mathscr{F}_0(Q)} \varSigma_\pm(x)
		= \frac{3}{4} \log Q + O(1).
	\]
\end{cor}
\begin{proof}
	From~\cref{eq:Length:Vs:OddPartialQuotientSum} and \cref{thm:Zhabitskaya:RestrictedVariant} we deduce that
	\[
		\frac{1}{\#\mathscr{F}(Q)}\sum_{x\in\mathscr{F}_0(Q)} \varSigma_{\text{odd}}(x)
		= c_1 (\log Q)^2 + c_2 \log Q + O(1).
	\]
	Moreover, by~\cref{eq:Length:Vs:EvenPartialQuotientSum},
	\[
		\sum_{x\in\mathscr{F}_0(Q)} \varSigma_{\text{even}}(x)
		= \sum_{x\in\mathscr{F}_0(Q)} \ell(1-x) + O(Q^2)
		= \sum_{x\in\mathscr{F}(Q)\setminus\mathscr{F}_0(Q)} \ell(x) + O(Q^2).
	\]
	On the other hand, \cref{Zhabitskaya} and \cref{thm:Zhabitskaya:RestrictedVariant} show that, after dividing by $\#\mathscr{F}(Q)$, the right hand side in the above is
	\[
		(C_1-c_1) (\log Q)^2 + (C_2-c_2) \log Q + O(1).
	\]
	In view of~\cref{eq:varSigmaPM:Decomposition}, the result follows from the previous considerations.
\end{proof}

\begin{proof}[Proof of \cref{thm:Goal}]
	Clearly it suffices to prove~\cref{eq:Ito:AsymptoticVersion}.
	To this end, observe that, by~\cref{Hickerson},  we have $D(x) = \varSigma_\pm(x) / 12 + O(1)$.
	Now~\cref{eq:Ito:AsymptoticVersion} follows immediately from this and \cref{cor:SigmaPmAsymptotics}.
\end{proof}

\section{Proof of \texorpdfstring{\cref{thm:Zhabitskaya:RestrictedVariant}}{Theorem\autoref{thm:Zhabitskaya:RestrictedVariant}}{}}
\label{sec:MainThm:SketchOfProof}
Before stating the key lemmas needed for the proof of \cref{thm:Zhabitskaya:RestrictedVariant}, we give a short informal sketch of the overall argument.
In~\cref{subsec:keylemmas} we state the three key lemmas we require.
The proof of \cref{thm:Zhabitskaya:RestrictedVariant} is given in~\cref{subsection:proofoftheorem}.

\subsection{Sketch of the proof}\label{subsec:sketchoftheproof}
In proving \cref{thm:Zhabitskaya:RestrictedVariant}, we adapt the approach of Zhabitskaya~\cite{zhabitskaya2009average-length}.
The idea, which goes back to Lochs~\cite{lochs1961statistik} and Heilbronn~\cite{heilbronn1968average-length}, is to transfer the problem of computing the (restricted) average of the lengths of (minus) continued fractions into a problem of counting lattice points inside certain regions.
By virtue of \cref{lem:TheSumOfLengths} and \cref{lem:reducedN4system} (below), the proof of \cref{thm:Zhabitskaya:RestrictedVariant} boils down to evaluating asymptotically the number of integer solutions of the system
\begin{equation}\nonumber
	\left\lbrace\begin{array}{@{}ll@{}}
		\gcd(p,q) = 1, &
		p,q\geq1, \\
		\inv_p(q)\leq p/2,\\
		2 \leq n q + kp \leq Q, &
		1\leq k<n.
	\end{array}\right.
\end{equation}
This amounts to counting the lattice points inside some region subject to some coprimality condition and the additional restriction $\inv_p(q)\leq q/2$.
The latter restriction is not present in~\cite{zhabitskaya2009average-length} and complicates the overall analysis.
Following~\cite{zhabitskaya2009average-length}, we split the problem of counting the solutions to the above system into five sub-cases.
For every case we have to count lattice points with certain properties inside regions (see~\cref{subsection:proofoftheorem} for the details).
This counting problem is solved in \cref{prop:CountingSolutionsToSystem} and it should be apparent from the proof of \cref{prop:CountingSolutionsToSystem} that the reason for the bias ($2c_2 > C_2$) in \cref{thm:Zhabitskaya:RestrictedVariant} is found within two of the considered cases.
More specifically, for one of these cases, the number of lattice points to be counted is given, up to some error term, by
\[
	\sum_{q<Q^{1/4}} \frac{1}{q} \sum_{\substack{q/2<b\leq q\\ \gcd(b,q)=1}}\frac{1}{q}\log \frac{Q^{1/2}}{q^2}
	= \sum_{q<Q^{1/4}} \frac{1}{q^2}\log \frac{Q^{1/2}}{q^2} \delta^{+}(q),
\]
where $\delta^{+}$ is the function appearing in \cref{lem:Euler'sPhiInHalfIntervals}.
The same procedure carried out for fractions greater than $1/2$ leads to the same expression with $\delta^{+}$ being replaced by $\delta^{-}$.
As \cref{lem:Euler'sPhiInHalfIntervals} shows, the functions $\delta^{+}$ and $\delta^{-}$ agree everywhere except at $1$ and $2$; this is the reason for $2c_2 > C_2$.

\subsection{Four lemmas}\label{subsec:keylemmas}
Each of the following lemmas plays a crucial r{\^o}le in the proof of \cref{thm:Zhabitskaya:RestrictedVariant}.
In fact, in spite of its simplicity, \cref{lem:inversiontrick}  turns out to be particularly useful in establishing \cref{prop:CountingSolutionsToSystem}: it permits a simple, yet important modification of the considered systems, allowing us to evaluate $R_3(U)$ and $R_5(U)$ (to be defined below) with the required precision (see~\cref{sec:NumberOfSolutions:RQ} for details).
The relevance of \cref{lem:Euler'sPhiInHalfIntervals} as the source of bias was already explained in \cref{subsec:sketchoftheproof}.
\cref{lem:TheSumOfLengths} and \cref{lem:reducedN4system} are adapted from \cite[Lemma~2 in \textsection~2.3]{zhabitskaya2009average-length} and allow us to translate our problem into the enumeration of the solutions of a system of inequalities (see \cref{eq:TheSystem:CC}).

\begin{lem}[Inversion trick]\label{lem:inversiontrick}
	Let $p,q\geq 2$ be two coprime integers. Then
	\[
		\inv_{p}(q)\leq \frac{p}{2}
		\quad\text{if and only if}\quad
		\inv_{q}(p)  >  \frac{q}{2}.
	\]
\end{lem}
\begin{proof}
	By coprimality, there are integers $a$ and $b$ such that $aq + bp = 1$, where $a = \inv_p(q) + tp$ and $b = \inv_q(p) + sq$ for some integers $s$ and $t$.
	Hence
	\[
		\inv_p(q)q + \inv_q(p)q - qp \equiv 1 \bmod pq.
	\]
	On the other hand, the left hand side of the above is contained in the interval $(-pq,pq)$.
	Hence, we conclude 
	\[
		\inv_p(q)q + \inv_q(p)p = 1 + pq,
	\]
	from which the lemma follows.
\end{proof}

\begin{lem}\label{lem:Euler'sPhiInHalfIntervals}
	Let $\varphi$ be Euler's totient function and define for every positive integer $q$ the counting functions
	\[
		\delta^-(q) = \sum_{\substack{ b \leq q/2 \\ \gcd(b,q)=1 }} 1
		\quad\text{and}\quad
		\delta^+(q) = \sum_{\substack{ q/2 < b \leq q \\ \gcd(b,q)=1 }} 1.
	\]
	Then the following assertions hold:
	\begin{enumerate}
		\item $\delta^+(1) = \delta^-(2) = 1$;
		\item $\delta^+(2) = \delta^-(1) = 0$;
		\item $\delta^+(q) = \delta^-(q) = \varphi(q) / 2$ for $q \geq 3$.
	\end{enumerate}
\end{lem}
\begin{proof}
	The assertions for $q\leq 2$ are trivial to check.
	For $q\geq 3$ note that the sets
	\[
		\set{ 1 \leq b \leq q/2 }[ \gcd(b,q)=1 ]
		\quad\text{and}\quad
		\set{ q/2 < b < q }[ \gcd(b,q)=1 ]
	\]
	are disjoint and in bijection by means of the map $b\mapsto q-b$.
	As the union of both sets contains exactly $\varphi(q)$ elements, we are done.
\end{proof}

\begin{lem}\label{lem:TheSumOfLengths}
	The sum $N_0(Q)$ of the lengths of the minus continued fraction expansions of the numbers $a/q$ with $1\leq a < q/2$, $q\leq Q$ is
	\[
		N_0(Q) = T_0(Q) + O(Q^2),
	\]
	where $T_0(Q)$ denotes the number of solutions $(a_1,q_1,a_2,q_2,m,n,a,b)\in\NN^8$ to the following system of equalities and inequalities:
	\begin{gather}\label{eq:TheSystem:1}
		\left\lbrace\begin{array}{@{}lll@{}}
			a_1q_2 - a_2 q_1 = 1, &
			1\leq a_1 \leq q_1, &
			1 \leq a_2 \leq  q_2/2, \\
			n a_2 - m a_1 = a, &
			n q_2 - m q_1 = b, &
			1 \leq a < b \leq Q, \\
			1 \leq m < n, &
			1\leq q_1 < q_2.
		\end{array}\right.
	\end{gather}
\end{lem}
\begin{proof}
	The claim follows \emph{mutatis mutandis} from~\cite[p.~1185--1186]{zhabitskaya2009average-length}.
\end{proof}

Next, discarding an acceptable number of solutions in the process, we reduce the system~\cref{eq:TheSystem:1} to a system with four variables. 

\begin{lem}\label{lem:reducedN4system}
	Let $R(Q)$ denote the number of solutions $(p, q, n, m) \in \NN^4$ of the system
	\begin{gather}\label{eq:TheSystem:CC}
		\left\lbrace\begin{array}{@{}ll@{}}
			\gcd(p,q) = 1, &
			p,q \geq 1, \\
			\inv_p(q) \leq p/2, \\
			2 \leq n q + kp \leq Q, &
			1 \leq k < n.
		\end{array}\right.
	\end{gather}
	Then, the number $N_0(Q)$ defined as in \cref{lem:TheSumOfLengths} satisfies
	\begin{align*}
		N_0(Q) = R(Q) + O(Q^2).
	\end{align*}
\end{lem}
\begin{proof}
	By virtue of \cref{lem:TheSumOfLengths}, we only need to show that $R(Q) = T_0(Q) + O(Q^2)$.
	It is convenient to exclude the solutions with $q_1=1$ from the discussion.
	We claim that their number is $O(Q^2)$ and, thus, negligible.
	To this end, consider first all the solutions of the system \cref{eq:TheSystem:1} with $q_1=1$.
	The conditions in system~\cref{eq:TheSystem:1} force that $a_1=a_2=q_1=1$ and $q_2=2$, reducing the system to 
	\[
		\left\lbrace\begin{array}{@{}ll@{}}
			n-m=a, &
			2n-m=b, \\
			1\leq a<b\leq Q, &
			1\leq m<n,
		\end{array}\right.
	\]
	for which one easily sees that its number of solutions is $\ll Q^2$.
	
	For the remainder of the proof we shall assume that $q_1\geq 2$.
	We claim that this assumption also implies that $a_1\leq q_1/2$.
	Indeed, suppose to the contrary that there was some solution to~\cref{eq:TheSystem:1} with $q_1\geq 2$ and $a_1>q_1/2$.
	We then deduce that
	\[
		2 = 2(a_1q_2-a_2q_1)
		\geq (q_1+1)q_2-2a_2q_1
		\geq (q_1+1)q_2-q_2q_1
		= q_2
		> q_1,
	\]
	in contradiction with $q_1 \geq 2$.
	
	Upon reducing the equation $a_1q_2-a_2q_1=1$ modulo $q_1$, we obtain $a_1=\inv_{q_1}(q_2)+tq_1$ for some integer $t$.
	As $a_1$ is positive and $q_1<q_2$, it follows that $t$ must vanish.
	Hence, $a_1=\inv_{q_1}(q_2)$.
	Consequently, $\inv_{q_1}(q_2)\leq q_1/2$.
	Now consider the system 
	\begin{equation}\label{eq:TheSystem:11}
		\left\lbrace\begin{array}{@{}lll@{}}
			\gcd(q_1,q_2) = 1, &
			1 \leq q_1 <q_2, & \inv_{q_1}(q_2)\leq q_1/2, \\
			2\leq n q_2 - m q_1 \leq Q, &
			1 \leq m < n.
		\end{array}\right.
	\end{equation}
	We now contend that the map $\Psi$ sending solutions $\boldsymbol{u} = (a_1,q_1,a_2,q_2,m,n,a,b)$ of~\cref{eq:TheSystem:1} with $q_1\geq 2$ to solutions $\boldsymbol{v} = (q_1, q_2, m, n)$ of~\cref{eq:TheSystem:11} (by means of dropping the entries $a_1$, $a_2$, $a$, and $b$) is a bijection.
	Indeed, above we have just seen that this map is well defined.
	To see that it is injective, suppose that $\boldsymbol{v}$ arises from some solution $\boldsymbol{u}$ of~\cref{eq:TheSystem:1}.
	As we have seen, $a_1 = \inv_{q_1}(q_2)$ is already determined by $\boldsymbol{v}$.
	But then, by $a_1q_2-a_2q_1 = 1$, also $a_2$ is determined by $\boldsymbol{v}$.
	Similarly, \cref{eq:TheSystem:1} then yields that also $a$ and $b$ are determined by $\boldsymbol{v}$, showing that $\Psi$ is injective.
	
	To show that $\Psi$ is also surjective, we start out with some solution $\boldsymbol{v} = (q_1, q_2, m, n)$ of~\cref{eq:TheSystem:11} and need to exhibit some preimage of $\boldsymbol{v}$ under $\Psi$.
	As $q_1$ and $q_2$ are coprime, there exist integers $a_1$ and $a_2$ such that $a_1q_2-a_2q_1 = 1$.
	Moreover, by replacing $(a_1,a_2)$ by $(a_1+tq_1,a_2+tq_2)$ with an appropriate integer $t$, we may assume that $0\leq a_1<q_1$.
	Furthermore, define $a = n a_2 - m a_1$ and $b = n q_2 - m q_1$.
	We now show that the octuple $\boldsymbol{u} = (a_1,q_1,a_2,q_2,m,n,a,b)$ is the desired preimage $\boldsymbol{v}$ under $\Psi$.
	We have shown above that $a_1 = \inv_{q_1}(q_2)$.
	Similarly, by reducing $a_1q_2-a_2q_1 = 1$ modulo $q_2$, we find that $a_2 = t_2 q_2 - \inv_{q_2}(q_1)$ for some integer $t_2$.
	We claim that $t_2 = 1$.
	To see this, first observe that
	\begin{equation}\label{eq:CongruenceModq1q2}
		a_1q_2-(q_2 - \inv_{q_2}(q_1))q_1
		\equiv a_1q_2-a_2q_1
		= 1
		\mod q_1q_2.
	\end{equation}
	From~\cref{eq:TheSystem:11} we see that $a_1 = \inv_{q_1}(q_2) \leq q_1/2$ and \cref{lem:inversiontrick} shows that $\inv_{q_2}(q_1) > q_2/2$.
	Therefore,
	\begin{equation}\label{eq:Estimate}
		a_1q_2-(q_2 - \inv_{q_2}(q_1))q_1
		\left\lbrace\begin{array}{@{}l@{}}
			> q_1q_2/2-(q_2 - q_2/2)q_1 = 0, \\
			< q_1q_2.
		\end{array}\right.
	\end{equation}
	Upon combining~\cref{eq:CongruenceModq1q2} and~\cref{eq:Estimate} we infer that the left hand side of~\cref{eq:Estimate} is equal to one and this shows that $a_2 = q_2 - \inv_{q_2}(q_1)$, as claimed.
	In particular, we have $a_2 < q_2/2$.
	Moreover~\cref{eq:TheSystem:11} shows that $b \leq Q$.
	It remains to show that $a < b$.
	We have
	\[
		q_1 a
		= q_1 ( n a_2 - m a_1 )
		= n (a_1q_2-1) - m a_1 q_1
		= a_1 (n q_2 - m q_1) - n
		= a_1 b - n.
	\]
	Using $a_1\leq q_1$, this shows that $a<b$.
	We conclude that $\Psi$ is surjective.
	
	Finally, we transform the system~\cref{eq:TheSystem:11} into the system~\cref{eq:TheSystem:CC} by changing the variables slightly by means of the following map:
	\begin{align*}
		\set{ \text{solutions }(q_1, q_2, m, n)\text{ of~\cref{eq:TheSystem:11}} } &
		\smash{{}\stackrel{1:1}{\longrightarrow}{}} \set{ \text{solutions }(p, q, k, n)\text{ of~\cref{eq:TheSystem:CC}} }, \\
		(q_1, q_2, m, n) &\longmapsto (q_1, q_2-q_1, n-m, m).
	\end{align*}
	This is easily checked to be a bijection; we omit the details.
\end{proof}

\subsection{Proof of \texorpdfstring{\cref{thm:Zhabitskaya:RestrictedVariant}}{Theorem\autoref{thm:Zhabitskaya:RestrictedVariant}}}\label{subsection:proofoftheorem}
In view of \cref{lem:reducedN4system}, it suffices to count the number of solutions of the system
\begin{equation}\label{eq:TheSystem:CCC}
	\left\lbrace\begin{array}{@{}ll@{}}
		\gcd(p,q) = 1, &
		p,q\geq1, \\
		\inv_p(q)\leq p/2, \\
		2 \leq n q + kp \leq Q, &
		1\leq k<n,
	\end{array}\right.
\end{equation}
with an error term of size $O\parentheses*{Q^2}$. 
The reader may notice the similarity between the system~\cref{eq:TheSystem:CCC} and the system \cite[Eq.~(42)]{zhabitskaya2009average-length}: they are almost identical, up to the additional constraints concerning coprimality and modular inversion. 
 Set $U=Q^{1/2}$ and consider the following five cases:
\begin{itemize}\label{enum:CasesExplanation}
	\item $p \leq q \leq U$;             \hfill (`Case~1')
	\item $p \leq q$, $U < q$;           \hfill (`Case~2')
	\item $q < p \leq U$;                \hfill (`Case~3')
	\item $q < p$, $U < p$, $n \leq U$;  \hfill (`Case~4')
	\item $q < p$, $U < p$, $U < n$.     \hfill (`Case~5')
\end{itemize}
Those cases are exactly the five cases appearing in \cite{zhabitskaya2009average-length}. 
The following proposition provides us the asymptotic number of solutions for each single case.

\begin{prop}\label{prop:CountingSolutionsToSystem}
	Suppose that $1\leq i\leq 5$ and let $R_i(U)$ denote the number of solutions to the system~\cref{eq:TheSystem:CCC} subject to the additional constraint that `Case~$i$' be satisfied.
	Then we have
	\begin{enumerate}
		\item $\displaystyle
			R_1(U)=\frac{\log2}{4\zeta(2)}U^4\log U+O\parentheses*{U^4}
		$,
		\item $\displaystyle
			R_2(U)=\frac{\log 2}{4\zeta(2)}U^4\log U+O\parentheses*{U^4}
		$,
		\item $\displaystyle
			R_3(U)=\frac{U^4\parentheses*{\log U}^2}{8\zeta(2)}+\frac{U^4\log U}{4\zeta(2)}\parentheses*{{\gamma}-\frac{\zeta'(2)}{\zeta(2)}+\frac{3\zeta(2)}{4}-\log 2}+O\parentheses*{U^4}
		$,
		\item $\displaystyle
			R_4(U)=\frac{U^4\parentheses*{\log U}^2}{8\zeta(2)}+\frac{U^4\log U}{4\zeta(2)}(\gamma-\log 2)+O\parentheses*{U^4}
		$,
		\item $\displaystyle
			R_5(U)=\dfrac{U^4}{4\zeta(2)}\parentheses*{\log U}^2+\dfrac{U^2}{2\zeta(2)}\parentheses*{\gamma-\dfrac{\zeta'(2)}{2\zeta(2)}-\dfrac{3}{2}{+\,\dfrac{3\zeta(2)}{8}}}\log U+O\parentheses*{U^4}
		$.
	\end{enumerate}
\end{prop}
The proof of \cref{prop:CountingSolutionsToSystem} is the most technical part of the paper.
We postpone it until~\cref{sec:NumberOfSolutions:RQ}.

\medskip

\pushQED{\qed}
Assuming the conclusion of \cref{prop:CountingSolutionsToSystem} for the moment, we are now in a position to finish the {\itshape proof of \cref{thm:Zhabitskaya:RestrictedVariant}}.
Indeed, by the above, we find that the number of solutions of the system~\cref{eq:TheSystem:CCC} is equal to
\[
	\dfrac{U^4}{2\zeta(2)}\parentheses*{\log U}^2+\dfrac{U^4}{2\zeta(2)}\parentheses*{2\gamma-\dfrac{\zeta'(2)}{\zeta(2)}-\dfrac{3}{2}{+\,\dfrac{3\zeta(2)}{4}}}\log U+O\parentheses*{U^4}.
\]
Substituting $U=Q^{1/2}$, we conclude for real numbers $Q>0$ which are not squares that
\begin{align}\label{N_OFinal}
	N_0(Q)
	=\dfrac{Q^2}{8\zeta(2)}\parentheses*{\log Q}^2+\dfrac{Q^2}{4\zeta(2)}\parentheses*{2\gamma-\dfrac{\zeta'(2)}{\zeta(2)}-\dfrac{3}{2}{+\dfrac{3\zeta(2)}{4}}}\log Q+O\parentheses*{Q^2},
\end{align}
where $N_0(Q)$ is the quantity described in \cref{lem:TheSumOfLengths}.
To obtain the same result in case $Q$ is a square, it suffices to notice that the asymptotic formula for $N_0(Q+1/2)$ matches~\cref{N_OFinal} up to an error of order $O(Q\log Q)$.
To finish the proof, we still have to restrict to the set $\mathscr{F}_0(Q)$. To this end, notice that by Möbius inversion we have
\begin{align*}
	\sum_{x\in\mathscr{F}_0(Q)} \ell(x) &
	=\mathop{\sum_{b\leq Q}\sum\limits_{a<b/2}}_{\gcd(a,b)=1} \ell\parentheses*{\frac{a}{b}}
	=\sum\limits_{d\leq Q}\mu(d)\mathop{\sum_{b\leq Q/d}\sum\limits_{a<b/2}} \ell\parentheses*{\frac{a}{b}} \\ &
	=\sum\limits_{d\leq Q}\mu(d)N_0\parentheses*{\dfrac{Q}{d}}.
\end{align*}
Hence, we deduce from \cref{MobInverFormula} and~\cref{N_OFinal} that
\[
	\sum_{x\in\mathscr{F}_0(Q)} \ell(x)=\dfrac{Q^2(\log Q)^2}{8\zeta(2)^2}+\dfrac{Q^2\log Q}{4\zeta(2)^2}\parentheses*{2\gamma-\dfrac{3}{2}-2\dfrac{\zeta'(2)}{\zeta(2)}+\,\dfrac{3\zeta(2)}{4}}+O\parentheses*{Q^2}.
\]
This concludes the proof of \cref{thm:Zhabitskaya:RestrictedVariant}.
\popQED

\section{Proof of \texorpdfstring{\cref{prop:CountingSolutionsToSystem}}{Proposition\autoref{prop:CountingSolutionsToSystem}}}
\label{sec:NumberOfSolutions:RQ}

As mentioned in \cref{subsection:proofoftheorem}, we count the solutions of \cref{eq:TheSystem:CCC} in five different cases which are exactly those considered by Zhabitskaya with the additional restrictions on coprimality and modular inversion.
Therefore, in what follows we often refer to the proof of~\cite[Theorem~2]{zhabitskaya2009average-length} as it contains several estimates which we employ directly here to simplify our exposition.

\subsection*{Case 1}
We count the number of solutions $R_1(U)$ of 
\begin{equation}\label{eq:TheSystem:CC1}
    \left\lbrace
        \begin{array}{@{}ll@{}}
            \gcd(p,q) = 1, &
            1\leq p\leq q\leq U, \\
            \inv_p(q)\leq p/2,\\
            2 \leq n q + kp \leq U^2, &
            1\leq k<n.
        \end{array}
    \right.
\end{equation}

If $p$ and $q$ are fixed, then the number of solutions of the above system with respect to the various $
1\leq k<n$ has been shown in \cite[(45)]{zhabitskaya2009average-length} to be equal to
\[
    \Sigma(p,q) \coloneqq \frac{U^4}{2q(p+q)}+E(U,p,q),
\]
where $E(U,p,q)$ is given explicitly in \cite[(45)]{zhabitskaya2009average-length}.
Thus, the number of solutions of \cref{eq:TheSystem:CC1} is equal to
\begin{align}\label{SolutionsCC1}
    \mathop{\sum_{q\leq U}\sum_{p\leq q}}_{\substack{\gcd(p,q)=1\\\inv_{p}(q)\leq p/2}}\Sigma(p,q)
    =\frac{U^4}{2}\mathop{\sum_{p\leq U}\sum_{p\leq q\leq U}}_{\substack{\gcd(p,q)=1\\\inv_{p}(q)\leq p/2}}\frac{1}{q(p+q)}
    +O\parentheses[\Big]{\mathop{\sum_{q\leq U}\sum_{p\leq q}}E(U,p,q)}.
\end{align}
The error term above has been proved in \cite[(45)--(47)]{zhabitskaya2009average-length} to be $O\parentheses*{U^3}$.
It remains to compute the first double sum in the right-hand side of~\cref{SolutionsCC1}.
We deal with the inner sum over $q$ first.
To this end, we set
\[
    f(x)=\frac{1}{x(p+x)},\quad g(x)=\frac{\varphi(p)}{2p}(x-p)\quad\text{and}\quad M(x)=\frac{x}{p^{1/2-\epsilon}}.
\]
Then \cref{eq:Abel'sSummation} and \cref{ModularHyperbola} yield that
\[
	\mathop{\sum_{p\leq q\leq U}}_{\substack{\gcd(p,q)=1\\\inv_{p}(q)\leq p/2}}\frac{1}{q(p+q)}
	\begin{multlined}[t]
		=\frac{\varphi(p)}{2p}\int_p^U\frac{\dd{x}}{x^2+xp} +{} \\
		\qquad\qquad + O_\epsilon\parentheses*{\frac{1}{p^{3/2-\epsilon}}+{\int_p^U\frac{x(2x+p)}{p^{1/2-\epsilon}(x^2+xp)^2}\dd{x}}} \\
		=\frac{\varphi(p)}{2p^2}{\int_{p}^U\parentheses*{\frac{1}{x}-\frac{1}{x+p}}\dd{x}} + O_\epsilon\parentheses*{p^{-3/2+\epsilon}} \hfill \\
		=\frac{\varphi(p)}{2p^2}\log2+O\parentheses*{U^{-1}} + O_\epsilon\parentheses*{p^{-3/2+\epsilon}}. \hfill\phantom.
	\end{multlined}
\]
We now take $\epsilon=1/3$ (any $\epsilon<1/2$ would do) and sum the above terms over $p\leq U$.
Our choice of $\epsilon$ ensures that the sum over the error terms remains bounded.
In view of \cref{lem:Euler'sPhiAsymptotics}~\cref{eq:Euler'sPhiAsymptotics3}, we conclude that
\begin{align}\label{eq:Case1Sum}
    \mathop{\sum_{p\leq U}\sum_{p\leq q\leq U}}_{\substack{\gcd(p,q)=1\\\inv_{p}(q)\leq p/2}}\frac{1}{q(q+p)}=\frac{\log 2}{2}\sum_{p\leq U}\frac{\varphi(p)}{p^2}+O(1)=\frac{\log 2}{2\zeta(2)}\log U+O(1).
\end{align}
For later use, observe also that the relation
\begin{align}\label{eq:RemarkforCase3}
    \mathop{\sum_{q< U}\sum_{q< p\leq U}}_{\substack{\gcd(p,q)=1\\\inv_{q}(p)>q/2 }}\frac{1}{p(q+p)}=\frac{\log 2}{2\zeta(2)}\log U+O(1)
\end{align}
can be derived in the same way as relation \cref{eq:Case1Sum} was.
Finally, upon combining~\cref{SolutionsCC1} with~\cref{eq:Case1Sum}, we conclude that
\begin{align*}
    R_1\parentheses*{U}=\frac{\log2}{4\zeta(2)}U^4\log U+O\parentheses*{U^4}.
\end{align*}

\subsection*{Case 2}
We count the number of solutions $R_2(U)$ of
\begin{gather}\label{eq:TheSystem:CC2}
    \left\lbrace
        \begin{array}{@{}lll@{}}
            \gcd(p,q) = 1, &
            1\leq p\leq q,&U<q, \\
            \inv_p(q)\leq p/2,\\
            2 \leq n q + kp \leq U^2, &
            1\leq k<n.
        \end{array}
    \right.
\end{gather}
In this case the inequalities $ n\leq{U^2}/{q}<U$ hold as well.

Let $\mathcal{C} \coloneqq \set{(p,q)\in\mathbb{N}^2}[\gcd(p,q)=1]$ and fix $k$ and $n$.
If $n+k\leq U$, then the domain of solutions of the above system  can be expressed as the lattice\footnote{The interested reader can have a look at the figures in \cite[pp.~1200]{zhabitskaya2009average-length} for a visual representation of those regions. 
The domain is the same but we restrict to its intersections with modular hyperbolas.}
\begin{align*}
    S_1(n,k)=\set*{(p,q)\in\mathcal{C}}[1\leq p\leq\frac{U^2}{n+k},\, U<q\leq\frac{U^2-kp}{n},\,\inv_p(q)\leq\frac{p}{2}]
\end{align*}
without the points of the lattice
\begin{align*}
    S_2(n,k)=\set*{(p,q)\in\mathcal{C}}[U<p\leq\frac{U^2}{n+k},\, U<q\leq p,\,\inv_p(q)\leq\frac{p}{2}].
\end{align*}
The number of integer points in $S_1(n,k)$ is equal to
\begin{equation*}
    \Sigma_1(n,k) \coloneqq \sum_{p\leq U^2/(n+k)}
    A_p\parentheses*{U,\frac{U^2-kp}{n}},
\end{equation*}
where $A_p(y,x)$ is defined in \cref{ModularHyperbola}.
Therefore, it follows that
\begin{align}\label{eq:Sigma1Case2}
    \begin{split}
        \Sigma_1(n,k)
        &=\sum_{p\leq U^2/(n+k)}\frac{\varphi(p)}{2p}\parentheses*{\frac{U^2}{n}-U-p\frac{k}{n}} +{} \\ & \qquad
        + \sum_{p\leq U^2/(n+k)}O_\epsilon\parentheses*{\frac{U^2-kp-nU+np}{np^{1/2-\epsilon}}}\\
        &=:S_{11}+S_{12}.
    \end{split}
\end{align}
Regarding the first sum,  \cref{lem:Euler'sPhiAsymptotics}~\cref{eq:Euler'sPhiAsymptotics1}--\cref{eq:Euler'sPhiAsymptotics2} and inequalities $k<n<U$ yield that
\[
    \begin{split}
        S_{11}
        &=\parentheses*{\frac{U^2}{n}-U}\parentheses*{\frac{U^2}{2\zeta(2)(n+k)}+O\parentheses*{\log\frac{U^2}{n+k}}}+{}\\
        &\quad-\frac{k}{n}\parentheses*{\frac{U^4}{4\zeta(2)(n+k)^2}+O\parentheses*{\frac{U^2}{n+k}\log\frac{U^2}{n+k}}}\\
        &=\frac{U^4}{2\zeta(2)n(n+k)}-\frac{U^3}{2\zeta(2)(n+k)}-\frac{kU^4}{4\zeta(2)n(n+k)^2}+O\parentheses*{\frac{U^2}{n}\log\frac{U^2}{n+k}}\\
        &=\frac{U^4}{4\zeta(2)n(n+k)}+\frac{U^4}{4\zeta(2)(n+k)^2}-\frac{U^3}{2\zeta(2)(n+k)}+O\parentheses*{\frac{U^2}{n}\log\frac{U^2}{n+k}}.
    \end{split}
\]
For the sum $S_{12}$ over the error terms, we estimate
\[
    S_{12}
    \ll_\epsilon\frac{U^2-nU}{n}\parentheses*{\frac{U^2}{n+k}}^{1/2+\epsilon}+\frac{n-k}{n}\parentheses*{\frac{U^2}{n+k}}^{3/2+\epsilon}\ll_\epsilon\frac{U^{3+2\epsilon}}{n(n+k)^{1/2+\epsilon}}.
\]
We work similarly for the number of integer points in $S_2(n,k)$:
\begin{align*}
    \Sigma_2(n,k) &
    =\sum_{U<p\leq U^2/(n+k)}A_p\parentheses*{U,p} \\ &
    =\sum_{U<p\leq U^2/(n+k)}\parentheses*{\frac{\varphi(p)}{2p}(p-U)+O_\epsilon\parentheses*{\frac{2p+U}{p^{1/2-\epsilon}}}}.
\end{align*}
Once more, \cref{lem:Euler'sPhiAsymptotics}~\cref{eq:Euler'sPhiAsymptotics1}--\cref{eq:Euler'sPhiAsymptotics2} and inequalities $k<n<U$ yield that
\begin{align*}
	\MoveEqLeft
	\sum_{U<p\leq U^2/(n+k)}\frac{\varphi(p)}{2p}(p-U) \\
	&=\frac{1}{4\zeta(2)}\parentheses*{\frac{U^4}{(n+k)^2}-U^2}+O\parentheses*{\frac{U^2}{n+k}\log\frac{U^2}{n+k}}+{}\\
	&\quad-\frac{U}{2\zeta(2)}\parentheses*{\frac{U^2}{n+k}- U+O\parentheses*{\log\frac{U^2}{n+k}}}\\
	&=\frac{U^4}{4\zeta(2)(n+k)^2}-\frac{U^3}{2\zeta(2)(n+k)}+O\parentheses*{U^2+\frac{U^2}{n}\log\frac{U^2}{n+k}},
\end{align*}
while for the sum of the error terms we obtain that
\begin{equation}\label{eq:ErrorS2Case2}
    \sum_{U<p\leq U^2/(n+k)}O_\epsilon\parentheses*{\frac{2p+U}{p^{1/2-\epsilon}}}
    \ll_\epsilon\sum_{U<p\leq U^2/(n+k)}p^{1/2+\epsilon}
    \ll_\epsilon\frac{U^{3+2\epsilon}}{(n+k)^{3/2+\epsilon}}.
\end{equation}
In view of \cref{eq:Sigma1Case2}--\cref{eq:ErrorS2Case2} and \cref{lem:AsymptoticFormulae}~\cref{eq:Asymptotic Formulae1}, we conclude that the number of solutions of the system \cref{eq:TheSystem:CC2} for pairs $(n,k)\in\mathbb{N}^2$ such that $1\leq k<n$ and $n+k\leq U$, is equal to
\begin{align*}
	\MoveEqLeft
    \mathop{\sum_{n<U}\sum_{k<n}}_{n+k\leq U}\parentheses*{\Sigma_1(n,k)-\Sigma_2(n,k)} \\
    &=\frac{U^4}{4\zeta(2)}\mathop{\sum_{n<U}\sum_{k<n}}_{n+k\leq U}\frac{1}{n(n+k)} + \mathop{\sum_{n<U}\sum_{k<n}}_{n+k\leq U}\brackets*{O\parentheses*{U^2}+O_\epsilon\parentheses*{\frac{U^{3+2\epsilon}}{nk^{1/2+\epsilon}}}}\\
    &=\frac{\log 2}{4\zeta(2)}U^4\log U+O\parentheses*{U^4}+O_\epsilon\parentheses*{U^{7/2+2\epsilon}}.
\end{align*}

Now we consider the pairs $(n,k)\in\mathbb{N}^2$ for which $1\leq k<n$ and $n+k>U$. 
In that case the number of solutions of the system \cref{eq:TheSystem:CC2} is smaller than the number of solutions of the same system without the restrictions on coprimality and modular inversion. 
This number has been computed in \cite[(54)--(56)]{zhabitskaya2009average-length} to be $O\parentheses*{U^4}$.
Therefore, by fixing $\epsilon\in(0,1/4)$, we obtain that
\begin{align*}
    R_2(U)=\frac{\log 2}{4\zeta(2)}U^4\log U+O\parentheses*{U^4}.
\end{align*}

\subsection*{Case~3}\label{subsec:case3}
We count the number of solutions $R_3(U)$ of
\[
    \left\lbrace
        \begin{array}{@{}ll@{}}
        \gcd(p,q) = 1, &
        1\leq q<p\leq U, \\
        \inv_p(q)\leq p/2,\\
        2 \leq n q + kp \leq U^2, &
        1\leq k<n.
        \end{array}
    \right.
\]
Similar as in Case~1 (see also \cite[(58)--(60)]{zhabitskaya2009average-length}), the number of solutions of the above system is equal to
\begin{align}\label{eq:TheSolutionsofCC3}
    \frac{U^4}{2}\mathop{\sum_{p\leq U}\sum_{ q<p}}_{\substack{\gcd(p,q)=1\\\inv_{p}(q)\leq p/2}}\frac{1}{q(p+q)}+O\parentheses*{U^3\log U}.
\end{align}
It remains to compute the double sum
\begin{align}\label{S1S2S3}
    \begin{split}
        \mathop{\sum_{p\leq U}\sum_{ q<p}}_{\substack{\gcd(p,q)=1\\\inv_{p}(q)\leq p/2}}\frac{1}{q(p+q)}
        &=\mathop{\sum_{p\leq U}\sum_{ q<p}}_{\substack{\gcd(p,q)=1\\\inv_{p}(q)\leq p/2}}\frac{1}{pq}-\mathop{\sum_{p\leq U}\sum_{q<p}}_{\substack{\gcd(p,q)=1\\\inv_{p}(q)\leq p/2}}\frac{1}{p(q+p)}\\
        &=\mathop{\sum_{p\leq U}\sum_{ p^{1/2}\leq q<p}}_{\substack{\gcd(p,q)=1\\\inv_{p}(q)\leq p/2}}\frac{1}{pq}+\mathop{\sum_{p\leq U}\sum_{q<p^{1/2}}}_{\substack{\gcd(p,q)=1\\\inv_{p}(q)\leq p/2}}\frac{1}{pq}-\mathop{\sum_{p\leq U}\sum_{q<p}}_{\substack{\gcd(p,q)=1\\\inv_{p}(q)\leq p/2}}\frac{1}{p(q+p)}\\
        &=:S_1+S_2-S_3.
    \end{split}
\end{align}

In view of \cref{lem:inversiontrick} and our remark \cref{eq:RemarkforCase3}, we have that
\begin{align}\label{S3}
    S_3
    =\mathop{\sum_{p\leq U}\sum_{q<p}}_{\substack{\gcd(p,q)=1\\\inv_{p}(q)\leq p/2}}\frac{1}{p(q+p)}
    =\mathop{\sum_{q< U}\sum_{q< p\leq U}}_{\substack{\gcd(p,q)=1\\\inv_{q}(p)>q/2 }}\frac{1}{p(q+p)}
    =\frac{\log 2}{2\zeta(2)}\log U+O(1).
\end{align}

Interchanging the sums in $S_1$ and applying \cref{lem:inversiontrick} yield that
\begin{align*}
    S_1=\sum_{q< U}\frac{1}{q}\mathop{\sum_{ q<p\leq  V_q}}_{\substack{\gcd(p,q)=1\\\inv_{q}(p)> q/2}}\frac{1}{p},
\end{align*}
where $V_q \coloneqq \min\set*{U,q^{2}}$.
If we set
\[
    f(x)=\frac{1}{x},\quad g(x)=\frac{\varphi(q)}{2q}(x-q)\quad\text{and}\quad M(x)=\frac{x}{q^{1/2-\epsilon}},
\]
then it follows from \cref{ModularHyperbola} and \cref{eq:Abel'sSummation} that 
\[
    \mathop{\sum_{q<p\leq  V_q}}_{\substack{\gcd(p,q)=1\\\inv_{q}(p)> q/2}}\frac{1}{p}
	\begin{multlined}[t]
		= \frac{\varphi(q)}{2q}\int_q^{V_q}\frac{\dd{x}}{x} + O_\epsilon\parentheses*{q^{-1/2+\epsilon}+\int_q^{V_q}\frac{\dd{x}}{xq^{1/2-\epsilon}}}\\
		= \frac{\varphi(q)}{2q}\log\frac{V_q}{q} + O_{\epsilon}\parentheses*{q^{-1/2+2\epsilon}}. \hfill\phantom.
	\end{multlined}
\]
Hence,
\[
	S_1 = \sum_{q<U^{1/2}}\frac{\varphi(q)}{2q^2}\log q+\sum_{U^{1/2}\leq q<U}\frac{\varphi(q)}{2q^2}\parentheses*{\log{U}-\log{q}}+\sum_{q<U}O_{\epsilon}\parentheses*{q^{-3/2+2\epsilon}}.
\]
We now take $\epsilon=1/5$, so that the last sum on the right hand side converges if $U$ is replaced by $\infty$ (any $\epsilon<1/4$ would do).
Therefore, in view of \cref{lem:Euler'sPhiAsymptotics}~\cref{eq:Euler'sPhiAsymptotics3}--\cref{eq:Euler'sPhiAsymptotics4}, we obtain that
\begin{align}\label{S1final}
    \begin{split}
        S_1
        &=\frac{\parentheses*{\log U}^2}{16\zeta(2)}+\frac{\parentheses*{\log U}^2}{4\zeta(2)}+O\parentheses*{\frac{\parentheses*{\log U}^2}{U}}-\frac{3\parentheses*{\log U}^2}{16\zeta(2)}+O(1)\\
        &=\frac{\parentheses*{\log U}^2}{8\zeta(2)}+O(1).
    \end{split}
\end{align}

Lastly, we proceed with the computation of $S_2$ where the bias in the \EAbyexcess{} makes its appearance for the first time.
Interchanging the sums in $S_2$ and applying \cref{lem:inversiontrick} yield that
\begin{equation}\label{S2}
	\begin{aligned}
	    S_2 &=
	    \mathop{\sum_{p\leq U}\sum_{q<p^{1/2}}}_{\substack{\gcd(p,q)=1\\\inv_{p}(q)\leq p/2}}\frac{1}{pq}
	    =\sum_{q< U^{1/2}}\frac{1}{q}\mathop{\sum_{q^{2}<p\leq U}}_{\substack{\gcd(p,q)=1\\\inv_{q}(p)> q/2}}\frac{1}{p} \\ &
	    =\sum_{q< U^{1/2}}\frac{1}{q}\mathop{\sum_{q/2<b\leq q}}_{\gcd(b,q)=1}\mathop{\sum_{ q^{2}<p\leq U}}_{p\equiv \inv_{q}(b)\bmod q} \!\!\! \frac{1}{p}.
	\end{aligned}
\end{equation}
Since
\[
	\#\set{ p\leq x }[ p\equiv \inv_{q}(b)\bmod q ] = \frac{x}{q} + \LandauO(1),
\]
for any coprime integers $1\leq b\leq q$, we know from \cref{eq:Abel'sSummation} that
\begin{align*}
    \mathop{\sum_{q^2<p\leq U}}_{p\equiv\inv_{q}(b)\bmod q}\frac{1}{p}=\frac{1}{q}\log \frac{U}{q^{2}}+O\parentheses*{q^{-2}}.
\end{align*}
Inserting this to \cref{S2} yields that
\begin{equation}\label{S21}
	\begin{aligned}
		S_{2}
		&=\sum_{q< U^{1/2}}\frac{1}{q}\mathop{\sum_{q/2<b\leq q}}_{\gcd(b,q)=1}\brackets*{\frac{1}{q}\log \frac{U}{q^2}+O\parentheses*{q^{-2}}} \\
		&=\sum_{q< U^{1/2}}\brackets*{\frac{\delta^+(q)}{q^2}\log \frac{U}{q^{2}}+O\parentheses*{\frac{\delta^+(q)}{q^3}}},
	\end{aligned}
\end{equation}
where $\delta^+(q)$ is defined in \cref{lem:Euler'sPhiInHalfIntervals}.

It is clear from relation \cref{S21} and \cref{lem:Euler'sPhiInHalfIntervals} where the bias occurs.
In the case we are considering (for fractions less than $1/2$), the terms which correspond to $q=1$ and $q=2$ come with weight $1$ and $0$, while in the complementary case (for fractions greater than $1/2$) where the counting function $\delta^+$ is replaced by $\delta^-$, they come with weight $0$ and $1/2$, respectively.	

Now in view of \cref{lem:Euler'sPhiInHalfIntervals}, \cref{lem:Euler'sPhiAsymptotics}~\cref{eq:Euler'sPhiAsymptotics3}--\cref{eq:Euler'sPhiAsymptotics4} we have that
\begin{align}\label{S2semi}
    \begin{split}
        S_2
        &=\log U+\sum_{ 3\leq q< U^{1/2}}{\frac{\varphi(q)}{2q^2}\log \frac{U}{q^{2}}}+O(1)\\
        &={\frac{1}{2}\log U-\frac{1}{8}\log U}+\sum_{ q< U^{1/2}}\frac{\varphi(q)}{2q^2}\parentheses*{\log {U}-2\log{q}}+O(1)\\
        &=\frac{3}{8}\log U+\frac{\parentheses*{\log U}^2}{8\zeta(2)}+\frac{\log U}{2\zeta(2)}\parentheses*{\gamma-\frac{\zeta'(2)}{\zeta(2)}}+O(1).
    \end{split}
\end{align}

Finally, we deduce from \cref{eq:TheSolutionsofCC3}, \cref{S1S2S3}, \cref{S3}, \cref{S1final} and \cref{S2semi} that 
\[
    R_3(U)
    =\frac{U^4\parentheses*{\log U}^2}{8\zeta(2)}+\frac{U^4\log U}{4\zeta(2)}\parentheses*{{\gamma}-\frac{\zeta'(2)}{\zeta(2)}+\frac{3\zeta(2)}{4}-\log 2}+O\parentheses*{U^4}.
\]

\subsection*{Case 4}
We count the number of solutions $R_4(U)$ of
\begin{gather}\label{eq:TheSystem:CC4}
    \left\lbrace
        \begin{array}{@{}lll@{}}
            \gcd(p,q) = 1, &
            1\leq q<p,& U<p, \\
            \inv_p(q)\leq p/2,\\
            2 \leq n q + kp \leq U^2, &
            1\leq k<n\leq U.
        \end{array}
    \right.
\end{gather}
Similar as in Case 2, we fix $k$ and $n$ and count the number of the above system, when $n+k\leq U$ and when $n+k> U$. 

If $n+k\leq U$, then the domain of solutions of \cref{eq:TheSystem:CC4} can be expressed as the union of the lattices\footnote{See \cite[pp.~1206]{zhabitskaya2009average-length} for figures.}
\begin{equation*}
    S_1(n,k)
    =\set*{(p,q)\in\mathcal{C}}[U< p\leq\frac{U^2}{n+k},\,1\leq q\leq p,\,\inv_p(q)\leq\frac{p}{2}]
\end{equation*}
and
\begin{align*}
    S_2(n,k)
    &=\set*{(p,q)\in\mathcal{C}}[\frac{U^2}{n+k}<p\leq\frac{U^2}{k},\,1\leq q\leq\frac{U^2-kp}{n},\,\inv_p(q)\leq\frac{p}{2}]\\
    &=\set*{(p,q)\in\mathcal{C}}[1\leq q\leq\frac{U^2}{n+k}-\theta,\, \frac{U^2}{n+k}<p\leq\frac{U^2-nq}{k},\,\inv_q(p)>\frac{q}{2}],
\end{align*}
where we have employed above \cref{lem:inversiontrick} and have introduced a parameter $\theta\in[0,1]$ which may vary.
The number of integer points in $S_1(n,k)$ is equal to
\begin{align*}
    \Sigma_1(n,k)
     \coloneqq \sum_{U<p\leq U^2/(n+k)}\mathop{\sum_{b\leq p/2}}_{\gcd(b,p)=1}\mathop{\sum_{ q\leq p}}_{q\equiv \inv_{p}(b)\bmod p}1
    =\sum_{U<p\leq U^2/(n+k)}\frac{\varphi(p)}{2}.
\end{align*}
It follows now from \cref{lem:Euler'sPhiAsymptotics}~\cref{eq:Euler'sPhiAsymptotics1} that
\begin{align}\label{Sigma1Case4}
    \begin{split}
        \Sigma_1(n,k)
        &=\frac{1}{4\zeta(2)}\parentheses[\bigg]{\parentheses*{\frac{U^2}{n+k}}^2-U^2}+O\parentheses*{\frac{U^2}{n+k}\log\frac{U^2}{n+k}}\\
        &=\frac{U^4}{4\zeta(2)(n+k)^2}+O\parentheses*{U^2+\frac{U^2}{n+k}\log\frac{U^2}{n+k}}.
    \end{split}
\end{align}
The number of integer points in $S_2(n,k)$ is equal to
\begin{align*}
    \Sigma_2(n,k)
     \coloneqq \sum_{q\leq{U^2}/(n+k)-\theta}
    B_q\parentheses*{\frac{U^2}{n+k},\frac{U^2-nq}{k}},
\end{align*}
where $B_q(y,x)$ is defined in \cref{ModularHyperbola}.
Upon applying said lemma, we infer that
\[
	\Sigma_{2}(n,k) = S_{21}+O_\epsilon(S_{22}),
\]
where
\begin{gather*}
	S_{21} = \sum_{q\leq{U^2}/(n+k)-\theta} \frac{\varphi(q)}{2q}\parentheses*{\frac{nU^2}{k(n+k)}-\frac{nq}{k}}, \\
	S_{22} = \sum_{q\leq{U^2}/(n+k)-\theta} \parentheses*{\frac{nU^2}{k(n+k)}-\frac{nq}{k}+q}q^{-1/2+\epsilon}.
\end{gather*}
From \cref{lem:Euler'sPhiAsymptotics}~\cref{eq:Euler'sPhiAsymptotics1}--\cref{eq:Euler'sPhiAsymptotics2} and inequalities $k<n<n+k\leq U$ we obtain that
\begin{align*}
	S_{21}
	&=\frac{nU^2}{2k(n+k)\zeta(2)}\parentheses*{\frac{U^2}{n+k}-\theta+O\parentheses*{\log\frac{U^2}{n+k}}}+{}\\
	&\quad-\frac{n}{4k\zeta(2)}\parentheses[\bigg]{\parentheses*{\frac{U^2}{n+k}-\theta}^2+O\parentheses*{\frac{U^2}{n+k}\log\frac{U^2}{n+k}}}\\
	&=\frac{nU^4}{4\zeta(2)k(n+k)^2}+O\parentheses*{\frac{nU^2}{k(n+k)}\log\frac{U^2}{n+k}}.
\end{align*}
For the sum over the error terms we estimate
\begin{align}\label{S22Case4}
    S_{22}
    \ll_\epsilon\frac{nU^2}{k(n+k)}\parentheses*{\frac{U^2}{n+k}}^{1/2+\epsilon}+\frac{n-k}{k}\parentheses*{\frac{U^2}{n+k}}^{3/2+\epsilon}
    \ll_\epsilon\frac{nU^{3+2\epsilon}}{k(n+k)^{3/2+\epsilon}}.
\end{align}
In view of \cref{Sigma1Case4}--\cref{S22Case4} and \cref{lem:AsymptoticFormulae}~\cref{eq:Asymptotic Formulae2}, we deduce that the number of solutions of the system \cref{eq:TheSystem:CC4} for pairs $(n,k)\in\mathbb{N}^2$ such that $1\leq k<n$ and $n+k\leq U$, is equal to
\begin{align*}
	\MoveEqLeft
    \mathop{\sum_{n<U}\sum_{k<n}}_{n+k\leq U} \parentheses*{\Sigma_1(n,k)+\Sigma_2(n,k)} \\
    &=\frac{U^4}{4\zeta(2)}\mathop{\sum_{n<U}\sum_{k<n}}_{n+k\leq U}\frac{1}{k(n+k)}+\mathop{\sum_{n<U}\sum_{k<n}}_{n+k\leq U}\brackets*{O\parentheses*{U^2}+O_\epsilon\parentheses*{\frac{U^{3+2\epsilon}}{kn^{1/2+\epsilon}}}}\\
    &=\frac{U^4\parentheses*{\log U}^2}{8\zeta(2)}+\frac{U^4\log U(\gamma-\log 2)}{4\zeta(2)}+O\parentheses*{U^4}+O_\epsilon\parentheses*{U^{7/2+2\epsilon}}.
\end{align*}

Now we consider the pairs $(n,k)\in\mathbb{N}^2$ for which $1\leq k<n$ and $n+k>U$. 
In that case the number of solutions of the system \cref{eq:TheSystem:CC4} is smaller than the number of solutions of the same system without the restrictions on coprimality and modular inversion. 
This number has been computed in \cite[(64)--(65)]{zhabitskaya2009average-length} to be $O\parentheses*{U^4}$.
Therefore, by fixing $\epsilon\in(0,1/4)$, we see that 
\[
    R_4(U)
    =\frac{U^4\parentheses*{\log U}^2}{8\zeta(2)}+\frac{U^4\log U}{4\zeta(2)}(\gamma-\log 2)+O\parentheses*{U^4}.
\]
\subsection*{Case 5}
We now count the number of solutions $R_5(U)$.
Employing \cref{lem:inversiontrick}, we find that this is the same as counting the number of solutions of the system
\begin{gather}\label{eq:TheSystem:CC52}
    \left\lbrace
        \begin{array}{@{}lll@{}}
            \gcd(p,q) = 1, &
            1\leq q<p,& U<p, \\
            \inv_q(p)> q/2,\\
            2 \leq n q + kp \leq U^2, &
            1\leq k<n,&U<n.
        \end{array}
    \right.
\end{gather}
Notice that the set of solutions of the above system is non-empty if, and only if, $k+q<U$.

For fixed $k$ and $q$ the number of solutions of \cref{eq:TheSystem:CC52} with respect to the various $n$ and $p$ is equal to
\begin{align*}
    \Sigma(k,q)
    &=\sum_{U<n\leq\parentheses*{U^2-k\lceil U\rceil}/q}\mathop{\sum_{q/2<b\leq q}}_{\gcd(b,q)=1}\mathop{\sum_{ U<p\leq \parentheses*{U^2-nq}/k}}_{p\equiv \inv_{q}(b)\bmod q}1\\
    &=\sum_{U<n\leq\parentheses*{U^2-k\lceil U\rceil}/q}\mathop{\sum_{q/2<b\leq q}}_{\gcd(b,q)=1}\parentheses*{\frac{1}{q}\parentheses*{\frac{U^2-nq}{k}-U}+O(1)}\\
    &=\sum_{U<n\leq\parentheses*{U^2-k\lceil U\rceil}/q}\parentheses*{\frac{\delta^+(q)}{q}\parentheses*{\frac{U^2-nq}{k}-U}+O\parentheses*{\frac{\delta^+(q)}{q}}},
\end{align*}
where $\ceil x \coloneqq \floor x+1$ is the ceiling function.
From \cref{lem:Euler'sPhiInHalfIntervals} and $k<U$ we deduce that
\begin{align*}
    \Sigma(k,1)
    &=\sum_{U<n\leq U^2-k\lceil U\rceil}
    \parentheses*{\frac{U^2}{k}-U-\frac{n}{k}}+O\parentheses*{U^2}\\
    &=\parentheses*{\frac{U^2}{k}-U}\parentheses*{U^2-k\ceil U-\floor U}+{}\\
    &\quad-\frac{\parentheses*{U^2-k\ceil U}^2+U^2-k\ceil U-{\ceil U}\floor U}{2k}+O\parentheses*{U^2}\\
    &=\frac{U^4}{2k}+O\parentheses*{U^3}
\end{align*}
and
$\Sigma(k,2)=0$.
Here is another case where the bias in the Euclidean algorithm appears.
Lastly, if $q\geq3$, then
\begin{align*}
    \Sigma(k,q)
    &=\sum_{U<n\leq \parentheses*{U^2-k\lceil U\rceil}/q}
    \frac{\varphi(q)}{2q}\parentheses*{\frac{U^2}{k}-U-\frac{nq}{k}}+O\parentheses*{U^2}\\
    &=\frac{\varphi(q)}{2q}\parentheses*{\frac{U^2}{k}-U}\parentheses*{\frac{U^2-kU+O(k)}{q}-U+O(1)}+O\parentheses*{U^2}+{}\\
    &\quad-\frac{\varphi(q)}{4k}\parentheses[\bigg]{\parentheses*{\frac{U^2-k U+O(k)}{q}+O(1)}^2- (U+O(1))^2}
\end{align*}
and by expanding each of the products we obtain that
\begin{align*}
    \Sigma(k,q)
    &=\frac{\varphi(q)}{2q}\parentheses*{\frac{U^4}{kq}-\frac{2U^3}{q}-\frac{U^3}{k}+\frac{kU^2}{q}+O\parentheses*{U^2}}+O\parentheses*{U^2}+{}\\
    &\quad-\frac{\varphi(q)}{4k}\parentheses*{\frac{U^4-2kU^3+k^2U^2}{q^2}+O\parentheses*{\frac{kU^2}{q^2}+\frac{U^2}{q}}-U^2+O(U)}\\
    &=\frac{\varphi(q)U^4}{4kq^2}-\frac{\varphi(q)U^3}{2q^2}-\frac{\varphi(q)U^3}{2qk}+\frac{\varphi(q)kU^2}{4q^2}+\frac{\varphi(q)U^2}{4k}+O\parentheses*{U^2}.
\end{align*}
Now we sum up over all pairs $(k,q)\in\mathbb{N}^2$ such that $k+q<U$, which is essentially equal to $R_5(U)$:
\begin{align*}
    \begin{split}
        \mathop{\sum_k\sum_q}_{k+q<U}\Sigma(k,q)
        &=\frac{U^4}{4}\parentheses[\bigg]{\mathop{\sum_k\sum_q}_{k+q<U}\frac{\varphi(q)}{kq^2}+\sum_{k\leq U-1}\frac{1}{k}-\sum_{k\leq U-2}\frac{1}{4k}}+O\parentheses*{U^4}+{}\\
        &\quad-\frac{U^3}{2}\mathop{\sum_k\sum_q}_{k+q<U}\parentheses*{\frac{\varphi(q)}{q^2}+\frac{\varphi(q)}{qk}}+\frac{U^2}{4}\mathop{\sum_k\sum_q}_{k+q<U}\parentheses*{\frac{\varphi(q)k}{q^2}+\frac{\varphi(q)}{k}}.
    \end{split}
\end{align*}
Each of the above sums is already given in \cref{lem:SomeMoreAsymptoticFormulae}, except of the harmonic sums
\[
    \sum_{k\leq U-1}\frac{1}{k}-\sum_{k\leq U-2}\frac{1}{4k}=\frac{3}{4}\log U+O\parentheses*{1}
\]
which have occurred here, because the quantities $\Sigma(k,1)$ and $\Sigma(k,2)$ are \emph{not} of the form
\[
    \frac{U^2\varphi(q)}{4kq^2}+O\parentheses*{U^3},\quad q=1,2,
\]
respectively.
Thus, we conclude that
\[
    R_5(U)
    =\frac{U^4\parentheses*{\log U}^2}{4\zeta(2)}+\frac{U^2\log U}{4\zeta(2)}\parentheses*{2\gamma-\frac{\zeta'(2)}{\zeta(2)}-3+\frac{3\zeta(2)}{4}}+O\parentheses*{U^4}.
\]


\subsection*{Acknowledgements}

It is the authors' pleasure to thank the anonymous referee for spotting some inaccuracies in an earlier draft and providing helpful suggestions.
During the preparation of this manuscript, the first-named author has been an associated student in the doctoral school programme `discrete mathematics' at Graz University of Technology.

\section{Statements and Declarations}

On behalf of all authors, the corresponding author states that there is no conflict of interest.
This version of the article has been accepted for publication in \emph{Mathematische Annalen}, after peer review, but is not the Version of Record and does not reflect post-acceptance improvements, or any corrections.
The \emph{Version of Record} is available online at:
{\small\url{https://dx.doi.org/10.1007/s00208-022-02452-2}}

\appendix

\section{Some asymptotic formulae}
\label{sec:AsymptoticFormulae}

We start by recalling, for the reader's convenience, a special case of a classical result on the distribution of points on modular hyperbolas.

\begin{lem}[Points on the modular hyperbola]\label{ModularHyperbola}
	Let $p$ be a positive integer and $x > y \geq 0$.
	Let
    \[
         A_p(y,x) \coloneqq \sum_{\substack{ y < q \leq x \\ \gcd(p,q)=1 \\ \inv_{p}(q) \leq p/2 }} 1
	     \quad\text{and}\quad
	     B_p(y,x) \coloneqq \sum_{\substack{ y < q \leq x \\ \gcd(p,q)=1 \\ \inv_{p}(q)   >  p/2 }} 1.
    \]
    Then, for any $\epsilon>0$,
    \begin{align*}
		A_p(y,x)
		= \dfrac{\varphi(p)}{2p}(x-y)
		+ \LandauO_\epsilon\parentheses*{\frac{x-y+p}{p^{1/2-\epsilon}}}
		= B_p(y,x).
    \end{align*}
\end{lem}
\begin{proof}
	This is a consequence of a more general folklore result about the points $(q,\inv_{p}(q))$ on a modular hyperbola (mod~$p$) where both coordinates are restricted to intervals.
	The interested reader may consult the survey~\cite{Shparmodularhyperbolas} (in particular, see Theorem~13 in \textsection~3.1 therein).
	A version with a slightly more explicit error term can be found, for instance, in~\cite[Lemma~1.7]{boca2000distribution-lattice-pts}.
	Strictly speaking, in both of the above sources, the intervals in question are restricted to have length not exceeding $p$.
	Nevertheless, the version required here easily follows from that by splitting $\lparen y,x\rbrack$ into $\ll 1+(x-y)/p$ intervals of length at most $p$.
\end{proof}

The next result is a version of Abel's summation formula.

\begin{lem}[Abel's summation formula]\label{eq:Abel'sSummation}
    Let $f,g \colon \ropeninterval{0,\infty}\to\RR$ be continuously differentiable functions.
    Let $y\geq 0$ be arbitrary.
    Suppose that $(a_n)_{n\in\NN}$ is a sequence of complex numbers such that the approximation
    \[
        \sum_{y < n \leq x} a_n
        = g(x) + \LandauO(M(x)),
    \]
    holds with some continuous function $M\colon \ropeninterval{0,\infty} \to \ropeninterval{1,\infty}$.
    Then
    \begin{align*}
        \sum_{y<n\leq x}a_nf(n)
        = {\int_{y}^x f(t) g'(t) \dd{t}} + \LandauO\parentheses*{
        \max_{t=x,y}\abs{f(t)M(t)}+ {\int_y^x \abs{f'(t)} M(t) \dd{t}}}.
    \end{align*}
\end{lem}
We also require the following lemma, which is an application of Möbius inversion.
\begin{lem}\label{MobInverFormula}
    Let $\Psi(Q) = a Q^2 (\log Q)^2 + b Q^2 \log Q + \LandauO(Q^2)$.
    Then
    \[
        \sum_{d\leq Q} \mu(d) \Psi\parentheses*{\dfrac{Q}{d}}
        = \frac{a}{\zeta(2)} Q^2 (\log Q)^2 + \frac{1}{\zeta(2)} \parentheses*{b-2a\dfrac{\zeta'(2)}{\zeta(2)}} Q^2 \log Q + \LandauO(Q^2).
    \]
\end{lem}

\begin{proof}
	For a proof see, e.g., \cite[Corollary~3]{zhabitskaya2009average-length}.
\end{proof}

We conclude with recording two technical lemmas which are used in the proof of \cref{prop:CountingSolutionsToSystem}.

\begin{lem}\label{lem:AsymptoticFormulae}
    The following asymptotic formulae hold for any $U\geq2$:
    \begin{enumerate}
        \item\label{eq:Asymptotic Formulae1}$\displaystyle
             \mathop{\sum_{n<U}\sum_{k<n}}_{n+k\leq U}\frac{1}{n(n+k)}=\log 2\log U+O(1)
        $,
        \item\label{eq:Asymptotic Formulae2}$\displaystyle
            \mathop{\sum_{n<U}\sum_{k<n}}_{n+k\leq U}\frac{1}{k(n+k)}=\frac{\parentheses*{\log U}^2}{2}+(\gamma-\log2)\log U+O(1)
        $.
    \end{enumerate}
\end{lem}

\begin{proof}
    For a proof see \cite[Lemma~9]{zhabitskaya2009average-length}. 
    Notice that the formulae there are being proved for $U\notin\NN$, but they are readily seen hold for $U\in\NN$ as well.
\end{proof}

\begin{lem} \label{lem:Euler'sPhiAsymptotics}
    The following asymptotic formulae hold for any $x\geq 2$:
    \begin{enumerate}
        \item\label{eq:Euler'sPhiAsymptotics1}$\displaystyle
            \sum_{q< x}\varphi(q)=\frac{x^2}{2\zeta(2)}+O(x\log x)
        $,
        \item\label{eq:Euler'sPhiAsymptotics2}$\displaystyle
            \sum_{q< x}\frac{\varphi(q)}{q}=\frac{x}{\zeta(2)}+O(\log x)
        $,
        \item\label{eq:Euler'sPhiAsymptotics3}$\displaystyle
            \sum_{q< x}\frac{\varphi(q)}{q^2}=\frac{1}{\zeta(2)}\parentheses*{\log x+\gamma-\frac{\zeta'(2)}{\zeta(2)}}+O\parentheses*{\frac{\log x}{x}}
        $,
        \item\label{eq:Euler'sPhiAsymptotics4}$\displaystyle
            \sum_{q< x}\frac{\varphi(q)}{q^2}\log q=\frac{\parentheses*{\log x}^2}{2\zeta(2)}+O(1)
        $.
    \end{enumerate}
\end{lem}
\begin{proof}
    The first two formulae are well known and the proof of the third one can be found in \cite[Corollary~4.5]{boca2007products-of-matrices}.
    The last formula can be deduced easily from \cref{eq:Euler'sPhiAsymptotics3} and \cref{eq:Abel'sSummation}.
\end{proof}

\begin{lem}\label{lem:SomeMoreAsymptoticFormulae}
	The following asymptotic formulae hold for any $U\geq2$:
	\begin{enumerate}
		\item$\displaystyle
			\mathop{\sum_k\sum_q}_{k+q<U}\frac{\varphi(q)}{kq^2}
			= \frac{\parentheses*{\log U}^2}{\zeta(2)} + \frac{\log U}{\zeta(2)}\parentheses*{2\gamma-\frac{\zeta'(2)}{\zeta(2)}}+O(1)
		$,
		\item$\displaystyle
			\mathop{\sum_k\sum_q}_{k+q<U}\frac{\varphi(q)}{q^2}
			= \frac{U\log U}{\zeta(2)} + O(U)
			= \mathop{\sum_k\sum_q}_{k+q<U}\frac{\varphi(q)}{qk}
		$,
		\item$\displaystyle
			\mathop{\sum_k\sum_q}_{k+q<U}\frac{\varphi(q)k}{q^2}
			= \frac{U^2\log U}{2\zeta(2)} + O\parentheses*{U^2}
			= \mathop{\sum_k\sum_q}_{k+q<U}\frac{\varphi(q)}{k}
		$.
	\end{enumerate}
\end{lem}
\begin{proof}
	They follow directly from the formulae of \cref{lem:Euler'sPhiAsymptotics} and the asymptotic formula of the truncated harmonic sum.
\end{proof}


\vfill%

\end{document}